%% file: main.tex
\documentclass[a4paper,11pt]{article}
\usepackage[T1]{fontenc}
\usepackage[utf8]{inputenc}
\usepackage[italian,english]{babel}
\usepackage{microtype}
\usepackage{quoting}
\usepackage{amsthm, amsmath, amssymb, amscd, mathrsfs, mathtools,tensor} 
\usepackage[a4paper]{geometry}
\usepackage{booktabs, caption, graphicx} 
\usepackage{enumitem}
\usepackage[framemethod=default]{mdframed}
\usepackage{braket}
\usepackage{tikz-cd}
\usepackage{centernot}
\usepackage{fancyhdr}
\usepackage{emptypage}
\usepackage{esint}
\usepackage[english]{varioref}

\mathtoolsset{showonlyrefs}

\usepackage[backend=biber,style=alphabetic]{biblatex}
\usepackage{csquotes}
\global\mdfdefinestyle{exampledefault}{}

\usepackage{settings} 

\author{Mattia Freguglia and Andrea Malchiodi\thanks{Scuola Normale Superiore, Piazza dei Cavalieri 7, 56126 Pisa. e-mails: mattia.freguglia@sns.it, andrea.malchiodi@sns.it} }

\date{}

\title{Yamabe metrics on conical manifolds}

\begin{document}

\maketitle

{\footnotesize
\begin{abstract}

\noindent We prove existence of Yamabe metrics on singular manifolds 
with conical points and conical links of Einstein type that include orbifold 
structures.  We deal with metrics of generic type and 
derive a counterpart of Aubin's classical result. Interestingly, the 
singular nature of the metric determines a different condition on the 
dimension, compared to the regular case. We derive asymptotic expansions 
on the Yamabe quotient by adding a proper and implicit lower-order correction to 
standard bubbles, whose contribution to the expansion of the quotient can be determined combining the decomposition of symmetric two-tensor fields and Fourier analysis on the conical links. 

\vspace{3ex}

\noindent{\it Key Words:} Singular Yamabe problem, Conical Metrics, 
Conformal Geometry.



\end{abstract}}

\input{intro}

\input{Section2}

\input{Section3}

\input{Section4A}

\input{Section4B}

\printbibliography

\end{document}

%% file: intro.tex
\section{Introduction}

Yamabe's problem consists in finding metrics in a given conformal class with constant scalar curvature, and represents a higher-dimensional counterpart of the 
classical uniformization problem. 
It was introduced in \cite{Yamabe-60} and it drew a lot of attention even after its 
resolution, that happened a few decades later.

One of the challenging features of the problem is its lack of compactness. 
In fact, letting $(M,g)$ denote a Riemannian manifold of dimension $n \geq 3$ 
with scalar curvature $R_g$, if one 
performs the conformal change of metric $\tilde{g} = w^{\frac{4}{n-2}} g$, with 
$w > 0$, then the scalar curvature $R_{\tilde{g}}$ of $\tilde{g}$ is given by 
\[
  \lp_g w = R_{\tilde{g}} w^{\frac{n+2}{n-2}}. 
\]
Here $\lp_g$ stands for the {\em conformal Laplacian}, whose expression is 
$$
  \lp_g w = - a \Delta_g w + R_g w, \qquad \quad a = \frac{4(n-1)}{n-2}, 
$$
and that satisfies the covariance law  
\begin{equation}\label{eq:conf-lap}
    \lp_{\tilde{g}} \phi  = w^{- \frac{n+2}{n-2}} \lp_g (w \, \phi), \qquad \phi \in C^\infty(M).  
\end{equation} 
In view of the above transformation law, the Yamabe problem consists in 
finding a positive solution to 
\begin{equation}\tag{$\mathcal{Y}$} \label{eq:Y}
	\lp_g u = \bar{R} \, u^{\frac{n+2}{n-2}}, 
\end{equation}
	for some $\bar{R} \in \R$. 
	
	The constant $\bar{R}$ plays the role of a Lagrange multiplier: indeed, one can consider extremals of the quadratic form 
	$$
	\int_M \Big( a |\nabla_g u|^2 + R_g u^2 \Big) \, d\mu_g = \int_M u \lp_g u \, d\mu_g
	$$
	constrained to the family of functions  normalized e.g. to one in $L^{\frac{2n}{n-2}}$. Constrained critical points, up to a constant 
	multiple, are easily seen to solve \eqref{eq:Y}. Equivalently, one can search for critical points 
	of the Sobolev-type quotient on $H^1(M) := \big\{ u \colon M \to \R \, | \, u, \nabla u \in L^2 \big\}$. 
	\begin{equation}\label{eq:Qg}
		Q_g (u) := \frac{\int_M \Big( a |\nabla_g u|^2 + R_g u^2 \Big) \, d\mu_g }{\Big( \int_M |u|^{\frac{2n}{n-2}} \, d\mu_g \Big)^{\frac{n-2}{n}}}, 
		\qquad \quad u \in H^1(M). 
	\end{equation}
	Similarly to $\lp_g$, the quotient $Q_g$ satisfies the relation $Q_{\tilde{g}} (u) = Q_g (w\, u)$ where $\tilde{g} = w^{\frac{4}{n-2}} g$. 
The difficulty in extremizing $Q_g$ is that the embedding $H^1(M) \hookrightarrow L^{\frac{2n}{n-2}}(M)$ is 
continuous but not compact, hence minimizing sequences for $Q_g$ might a priori not converge, 
and develop instead {\em bubbling}, i.e. exhibit concentration of the $H^1$-norm or of the $L^{\frac{2n}{n-2}}$-norm 
at arbitrarily small scales. This phenomenon results in the nonlinear term on the  right-hand side of \eqref{eq:Y} being 
of {\em critical type}.

Progress on the problem was made by considering the {\em Yamabe quotient}, which is 
defined as 
\[
   \mathcal{Y}(M, g) := \inf_{\substack{{u \in C^{\infty}(M)} \\ u > 0}} Q_g(u). 
\]
The above covariance of $Q_g$ implies that indeed $\mathcal{Y}(M, g) = \mathcal{Y}(M, \tilde{g})$ for any metric 
$\tilde{g}$ conformal to $g$, and therefore we will use the notation $\mathcal{Y}(M,[g])$, where $[g]$ 
stands for the conformal class of $g$. Conformal classes $[g]$ for which $\mathcal{Y}(M,[g]) > 0$ (respectively, 
$= 0$  or $< 0$) are said to be of {\em positive Yamabe class} (respectively, of {\em null} or {\em negative} class). 
In \cite{Trudinger-68} it was shown that for every $n \geq 3$ there exists a dimensional constant $\varepsilon_n > 0$ such that
$\mathcal{Y}(M,[g])$ is attained provided $\mathcal{Y}(M,[g]) \leq \eps_n$, which holds in particular  when the 
Yamabe class is negative or null. A sharper version of this result was obtained in \cite{Aubin-76-2}, 
where it was proven that the infimum of~$Q_g$ is attained whenever 
\begin{equation}\label{eq:Y<}
  \mathcal{Y}(M,[g]) < \mathcal{Y}(S^n,[g_{S^n}]), 	
\end{equation}
where $g_{S^n}$ denotes the round metric on the sphere. As $(S^n,g_{S^n})$ is conformally 
equivalent to $\R^n$ via stereographic projection, $\mathcal{Y}(S^n,[g_{S^n}])$ coincides up to a 
dimensional constant with the Sobolev constant of $\R^n$, which is attained by the 
function 
\begin{equation}\label{eq:U-def}
\bar{U}(x) = \bar{U}(\abs{x}) = \left( \frac{1}{1+|x|^2} \right)^{\frac{n-2}{2}}, 
\end{equation}
see also \cite{Tal}. Such profiles, after a proper {\em dilation} of type $\bar{U}_\eps(x) 
\cong \eps^{-\frac{n-2}{2}} \bar{U}(x/\eps)$, for $\eps$ small, can be glued using 
normal coordinates to any point $p$ of $M$, and the geometry of $(M,g)$ affects 
the expansion of $Q_g(\bar{U}_\eps)$ for $\eps$ small. Such an expansion is 
ruled by the Weyl tensor $W_g$ at $p$ when $n \geq 6$, making $Q_g(\bar{U}_\eps)$ smaller 
than $\mathcal{Y}(S^n,[g_{S^n}])$.  Instead, for $n \leq 5$ (or when $g$ is 
locally conformally flat), it is 
determined by the {\em mass} of $(M,g)$ conformally deformed by the Green's function 
of $\lp_g$ with pole at $p$.  Using the  {Positive Mass Theorem} from~\cite{Schoen-Yau-79},~\cite{Schoen-Yau-81} and~\cite{Schoen-Yau-88}, 
R.~Schoen verified in~\cite{Schoen-84} the above inequality also for the latter cases. 
Notice that the decay of $\bar{U}$ is faster in higher dimensions, and 
therefore the expansions are of local nature in this case.

\

In this paper we consider a singular version of the Yamabe problem, and we are mainly 
(but not only, see Remark \ref{rem:intro} (a))
interested in the case of manifolds 
possessing a finite number of conical points. More precisely, these spaces are compact metric spaces $(M,d)$ such that there exist a finite number of points $\{ P_1, \dots , P_l \}$ and a Riemannian metric $\bar{g}$ on $M \setminus \{ P_1, \dots , P_l \}$ that induces the same distance $d$. In addition, we ask that for every point $P_i$ there exists an open neighborhood $\mathcal{U}_i$ of $P_i$ and a diffeomorphism $\sigma_i$ such that
\begin{equation} \label{eq:diffeo}
    \sigma_i \colon \mathcal{U}_i \setminus \{ P_i \} \to (0,R_i) \times Y_i, \quad (\sigma_i)_{*} \bar{g} = ds^2 + s^2 h_i(s),
\end{equation}
where $Y_i$ is a smooth closed manifold called the \emph{link} over the \emph{conical} point $P_i$ and $h_i(s)$ is a smooth family of metrics on $Y_i$.
Therefore, the metric ball of radius $s$ around 
a conical point $P$ is of the type  
\[
B_s(P) = \big( [0,s) \times Y \big)|_\sim,
\]
namely a topological cylinder collapsed on one side. The metric $\bar{g}$ behaves like 
\[
  \bar{g} \simeq ds^2 + s^2 h_0, \quad \text{as $s \to 0$},  
\]
where $s$ denotes the geodesic distance from $P$ and $h_0$ is the metric on $Y$ given by $h(s)$ at $s=0$. We refer e.g. to~\cite{stone} for the use of normal coordinates in the neighborhood 
of conical points.

There has been recently a growing interest in non-smooth manifolds, because 
they might for example arise as Gromov-Hausdorff limits of degenerating regular ones. 
Singularities of the form described above appear in the study of Einstein or 
\emph{critical} metrics (\cite{And}, \cite{Ti-Vi-1}, \cite{Ti-Vi-2}): 
in these cases singular points are of \emph{orbifold type}, 
i.e. for which  
\begin{equation}\label{eq:Y-gamma}
  Y = S^{n-1}/\Gamma, 	
\end{equation}
where $\Gamma$ stands for a finite group of isometries of $S^{n-1}$ without 
fixed points. Other types of isolated singularities appear e.g. in the extremization of the Yamabe quotient with respect to the conformal class $[\bar{g}]$ 
(see \cite{Ak-94} and \cite{Ak-96}), but the structure of general singular manifolds can be quite rich: we refer the 
reader to   \cite{Maz-edge} for a more systematic treatment of stratified spaces, 
which also played a fundamental role in the study of K\"ahler-Einstein metrics, see 
e.g.~\cite{JMR} and references therein. Some positive mass theorems in the singular 
setting have been proven in e.g. \cite{Nak}, \cite{Ak-Bo} and  \cite{DSW}.

The Yamabe problem on singular spaces can  present new phenomena and difficulties 
compared to the regular case. The authors in \cite{Ak-Mon} showed that on $S^n$,
with a conical wedge being an equatorial $S^{n-2}$ of angle $\alpha \geq 4 \pi$, 
the Yamabe quotient is not attained. In 
\cite{Vi} there are examples of conformally orbifold-compactified 
(four-dimensional) hyperk\"ahler ALE manifolds for which the Yamabe equation is 
not even solvable, at least with conformal factors preserving a 
smooth orbifold structure. This example is somehow reminiscent 
of the two-dimensional \emph{teardrop} - a topological sphere with a conical point -  
on which no metric with constant Gaussian curvature exists.

One feature of the singular Yamabe problem  is that blow-ups might occur 
either at regular points or at singular ones. While in the former case the local contribution 
to the Yamabe quotient is $\mathcal{Y}(S^n, [g_{S^n}])$, as for smooth manifolds, in the 
latter case the value might be different and given by the {\em local Yamabe constant} 
defined as 
\begin{equation}\label{eq:YYPP}
  \mathcal{Y}_P = \lim_{r \to 0} \big( \inf \big\{ Q_{\bar{g}}(u) \colon u \in W^{1,2}_0(B_r(P)) \big\} \big). 
\end{equation}

The local Yamabe constant $\mathcal{Y}_P$ coincides with the Yamabe constant of the purely conical metric
$ds^2 + s^2 h_0$ on $(0,+\infty) \times Y$,  
is upper bounded by $\mathcal{Y}(S^n,[g_{S^n}])$
and is explicitly known in some examples. For orbifold points 
 $\mathcal{Y}_P = \mathcal{Y}(S^n,[g_{S^n}])/k^{2/n}$, with $k$ being the cardinality of the group 
$\Gamma$ in~\eqref{eq:Y-gamma}, see \cite{Ak-orb}, and if $(Y,h_0)$ is an Einstein manifold with 
$\mathrm{Ric}(h_0) = (n-2) h_0$, one has that $\mathcal{Y}_P = ( \mathrm{Vol}_{h_0} (Y) / \mathrm{Vol}_{g_{S^{n-1}}}(S^{n-1}))^{\frac 2n} \mathcal{Y}(S^n,[g_{S^n}])$, see 
Corollary~1.3 in~\cite{Petean}.

In \cite{Ak-Bo} and \cite{ACM}   the authors proved, in the spirit of
 \cite{Aubin-76-2}, that  solutions of the  Yamabe equation 
exist on certain stratified manifolds, recalled later in more detail, provided one has the inequality  
\begin{equation} \label{eq:sing-quot}
  \mathcal{Y}(M,[\bar{g}]) < \min_{P \in S} \mathcal{Y}_P, 
\end{equation}
where $S$ denotes the family of singular points. Similarly to the regular case the right-hand side of the previous formula is defined as
\begin{equation}
    \mathcal{Y}(M,[\bar{g}])
    :=
    \inf_{\substack{{u \in W^{1,2}(M)} \\ u \neq 0}} Q_{\bar{g}}(u)
    =
    \inf_{\substack{{u \in W^{1,2}(M)} \\ u \neq 0}} \frac{\int_{\Omega} \Big( a |\nabla_{\bar{g}} u|^2 + R_{\bar{g}} u^2 \Big) \, d\mu_{\bar{g}} }{\Big( \int_{\Omega} |u|^{\frac{2n}{n-2}} \, d\mu_{\bar{g}} \Big)^{\frac{n-2}{n}}}, 
\end{equation}
where $\Omega$ is the \emph{regular part} of $M$ (in the case of a conical manifold is simply the complement of the conical points)
and $W^{1,2}(M)$ denotes the space obtained as the closure of Lipschitz functions with respect to the $W^{1,2}$-norm.

\smallskip

To our knowledge, even though different examples are known (see also \cite{Jeffres-Rowlett} for cases with negative scalar curvature), there are however no general conditions for which  \eqref{eq:sing-quot} 
 is verified, and it is indeed the purpose of this paper to have one. Before stating our 
 results, we recall the notion of {\em normalized Einstein-Hilbert functional}:
 \begin{equation} \label{eq:def-EH}
  \mathcal{R}(g) = \mathcal{R}_M({g}) := \mathrm{Vol}_{g}(M)^{-\frac{n-2}{n}} \int_M R_{{g}} \, d\mu_{{g}}.
 \end{equation}
 Recalling also \eqref{eq:Qg}, it turns out that if $\tilde{g} = u^{\frac{4}{n-2}} g$, then $Q_g(u) = \mathcal{R}(\tilde{g})$. 
 Critical points of~$\mathcal{R}$ restricted to a given conformal class are Yamabe metrics, while  free critical points 
are Einstein metrics. An Einstein metric is said to be {\em strictly stable} if the second variation of~$\mathcal{R}$ is 
negative-definite in the $tt$-variations of the metric (see \eqref{eq:R-parte-tt} and the discussion before it):  standard spheres and their quotients 
are well-known to be stable, and other examples are discussed in 
Remark \ref{rem:intro}~(a).

\

Before stating our result, we clarify our assumption on the metric near conical points, which is of generic type:  
\begin{itemize}
		\item[$(H_P)$] If $P$ is a conical point with link $Y$ and $s$ is the geodesic distance 
	from   $P$, we assume that for $s$ small 
	the metric is of the type 
	\begin{equation}\label{eq:g-assumption}
		g = ds^2 + s^2 h(s), 
	\end{equation}
	where $h(s)$ is a one-parameter family of metrics on $Y$ that depends regularly 
	on $s$, see the comments at the beginning of Subsection \ref{ss:variations} for 
 the precise condition. We denote by $\xi$ the first-order term in $s$ in the expansion 
	of $h(s)$, namely: 
	\begin{equation}\label{eq:expansion-h}
		h(s) =  h_0 + s \xi + O(s^2). 
	\end{equation}
\end{itemize}

 As for \cite{Aubin-76-2}, we aim to find an asymptotic expansion of the 
 Yamabe quotient that depends on conformally invariant quantities, so it is 
 relevant for us to understand the behavior of the expansion in \eqref{eq:expansion-h} 
 under conformal changes of metric. Referring to Section \ref{s:exp-I} for details, 
 we introduce the condition 
\begin{equation}\label{s:xi-conf} \tag{$\xi_P$}
	\xi \neq \nabla_{h_0}^2 f + f h_0, \quad \text{for all smooth $f \colon Y \to \R$.} 
\end{equation}
 For suitable conformal deformations of the metric $g_0 := ds^2 + s^2 h_0$, in geodesic 
 coordinates the term $\xi$ is precisely of the form $\nabla^2_{h_0} f + f h_0$, 
 with $f$ depending on the conformal factor, 
 and hence condition \eqref{s:xi-conf} {\em distinguishes at first order 
 our metric from conformally deformed purely conical ones}, see 
 \eqref{eq:transf-xi} and the preceding discussion.

\smallskip

This is our main result.

\begin{theorem}\label{t:main}
 	 Let $(M^n, \bar{g})$, $n \geq 4$, be  compact  	 
 	   with finitely-many conical points $\{P_1, \dots, P_l\}$ 
 	 such that for all $i$ the metric $\bar{g}$ satisfies $(H_{P_i})$
    with $\mathrm{Ric}(h_{i,0}) = (n-2) h_{i,0}$. 
 	 Suppose there is a conical point $P$ whose link $(Y,h_0) \neq (S^{n-1}, g_{S^{n-1}})$ is strictly stable, satisfies $(\xi_P)$ and for which 
 	 \begin{equation}\label{eq:P-min}
 	 	 	   \mathcal{Y}_P = \min_{i = 1, \dots, l} \mathcal{Y}_{P_i}.  
 	 \end{equation}
 	 Then $(M,\bar{g})$ admits a Yamabe metric. 
 \end{theorem}

Before discussing the proof, some comments  are in order. 

\begin{remark} \label{rem:intro}
	(a) The result applies to manifolds with finitely-many 
	 conical points  of orbifold type with structure as in $(H_P)$, in which case our assumptions are generic. 
	Study of the stability of Einstein metrics and other examples of conformally stable Einstein metrics  can be  found e.g. in 
	\cite{Cao-He}, \cite{Koi},  \cite{Kro1} and \cite{Sem-We}. 
	 
	 (b) The condition $\mathrm{Ric}(h_0) = (n-2) h_0$ on the links (or, more in general, that their scalar curvature satisfies $R_{h_0} = (n-1)(n-2)$) implies that the scalar curvature of the 
	 (purely) conical metric $ds^2 + s^2  h_0$ vanishes identically, but for metrics 
	 as in $(H_P)$ it might diverge at rate $1/s$ near $P$,
	 see Section \ref{s:exp-I}. This makes anyway $R_{\bar{g}}$ in $L^p(M,g)$ 
	 for any $p < n$. Therefore, our manifolds satisfy assumptions i)-iv) 
  in Section 1 of \cite{ACM}: in fact, our proof applies with minor changes to their  
	 class as well,  provided there is a conical point $P$ as 
	 above for which $\mathcal{Y}_P = \min_{Q \in S} Y_Q$, where $S$ stands  for the 
	 set of singular points in $M$. In \cite{ACM} (see also \cite{Bo-Pre}) the regularity of Yamabe conformal 
	 factors is studied as well, and for our case it is of H\"older type near the 
	 conical points, see Proposition \ref{p:holder}.

	(c) Our assumption on the dimension differs from that in \cite{Aubin-76-2}: 
	this fact is due to the singular structure of the metric, which generates lower-order 
	corrections in the expansion of the Yamabe quotient. 
	
	(d) Conical metrics as in \eqref{eq:g-assumption} are conformal to cylindrical ones, so 
	our result could be restated for manifolds with cylindrical ends under proper conditions 
	on the asymptotic behavior of the metric. Results about existence, rigidity  and on the Yamabe 
	quotient for such manifolds can be found in e.g \cite{Ak-Bo}, \cite{Ak-Bo-L2}. 
	
	(e) In the non-existence example  \cite[Thoerem 1.3]{Vi} one has a single orbifold point with $\xi = 0$, 
	and hence our result does not apply to this case.  See also \cite{Ju-Vi} 
 for a related scalar curvature prescription problem. 
\end{remark}

The strategy of the proof relies as in \cite{Aubin-76-2} and \cite{Schoen-84} 
on finding a competing conformal factor that guarantees \eqref{eq:sing-quot}, 
which by the assumption on $P$ amounts to 
showing  $\mathcal{Y}(M,[\bar{g}]) < \mathcal{Y}_P$: however, some crucial differences occur due to the 
singular structure of the metric.

As noticed in \cite{Lee-Parker-87}, 
in the regular case the use of normal coordinates, and in particular {\em conformal normal coordinates} (i.e. normal coordinates after 
a conformal change of metric), 
is extremely useful in order to work out the expansion of the Yamabe quotient 
for highly concentrated functions. In fact, for such coordinates the volume element near the concentration point 
can be made flat at an arbitrary order, as well as for the scalar curvature up to first order. 
Concerning second-order expansion of $R_g$ instead, one has that the squared norm of the Weyl 
tensor appears as minus the Laplacian of $R_g$, showing that for some $C_n > 0$
$$
  Q_g(\bar{U}_\eps) = \mathcal{Y}(S^n,[{g_{S^n}}]) - C_n \eps^4 {|W_g|^2} + o(\eps^{4}), \quad 
  \text{as $\eps \to 0$,}  
$$
see the notation after \eqref{eq:U-def}, with an extra logarithmic term for $n=6$ that appears in front of the second term on the right-hand side. 

In our case, regarding the function $\bar{U}_\eps$ as an approximate solution of the 
Yamabe equation, we have less precision due to the presence of first-order terms in $s$ 
appearing in~\eqref{eq:expansion-h}, affecting the expansion of $Q_{\bar{g}}$ in a 
relevant way. 
As a result, we need to add a properly chosen correction $\eps \Psi_\eps$ to $\bar{U}_\eps$ (see~\eqref{eq:equation-Psi}-\eqref{eq:L1copy}) with the same dilation 
parameter but of order $\eps$ in the $W^{1,2}$-norm. While this correction is not needed for regular metrics in conformal normal coordinates, 
for us in generic situations it is \underline{non-zero}, it \underline{cannot be eliminated via conformal changes} (see Remark \ref{r:annull-phi1}), and it \underline{crucially 
contributes to the expansion of the Yamabe quotient}. Referring to 
Proposition \ref{p:exp-full} for details, if $\xi$ is as in \eqref{eq:expansion-h} we find that 
\begin{multline} \label{eq:2nd}
    Q_{\bar{g}}(\bar{U}_\eps + \eps \Psi_\eps) = \mathcal{Y}_P \\ + \eps^2 \bigg( 
    \frac{\omega}{2} \mathrm{Vol}_{h_0}(Y)^{\frac{n-3}{n-1}} \mathcal{R}''(h_0)[\xi,\xi]
    + \frac{\omega}{4} \int_Y \Big( (\mathrm{tr} \xi )^2 - \abs{\xi}^2 \Big) \, d \mu_{h_0} 
    - \int_{\mathcal{C}(Y)} \Psi \widetilde{\lp} \Psi \, d\mu_{g_0} \bigg)
    + o(\eps^2), 
\end{multline}
where $\widetilde{\lp}$ is the linearized Yamabe operator at $\bar{U}$ (see~\eqref{eq:V-ker-L}), 
$\omega$ is a positive dimensional constant, and an extra factor of order $|\log \eps|$ in front of 
the second-last term when $n = 4$. Here $\mathcal{C}(Y)$ stands for the metric cone over $Y$, with metric 
$g_0$  given after \eqref{s:xi-conf}, see also Subsection~\ref{ss:FAYQ}.

Another difficulty we encounter at this point is that, while the correction $\Psi$ can be found by abstract 
means, we cannot generically write  $\Psi$ in closed form. More precisely, using the standard decomposition of the metric variation $\xi$ 
into its conformal, Lie-derivative and $tt$ parts, we 
can isolate the corresponding contributions by $\Psi$ in~\eqref{eq:2nd}. 
In doing so, we suitably split  $\Psi = \Psi_1 + \Psi_2$, with only $\Psi_1$ explicit (see the discussion in Subsection~\ref{ss:role}), 
and we use for the contribution of $\Psi_2$ in \eqref{eq:2nd} a Fourier decomposition in the link. 
Using a monotonicity property of certain integral quantities with 
respect to the eigenvalues of the Laplacian on the link, we  show that the sum of the terms  within  round brackets 
of~\eqref{eq:2nd} is indeed negative under our assumptions.

The four-dimensional case has to be treated separately, since it leads to diverging integrals, with an extra logarithmic term in the scaling parameter $\eps$. In order to treat this issue,  we replace the correction $\Psi_2$ with a solution of the Laplace 
equation on the limit cone instead of the linearized Yamabe equation. While the accuracy of the resulting approximate solution $\bar{U}_\eps + \eps \Psi_\eps$ gets worse  near the conical tip, it is preserved at scales between $\eps$ and $1$, which is sufficient for us to estimate the logarithmic divergences.

\begin{remark} 
(a) Let the functions $f, G : Y \to \R$ be as in~\eqref{eq:decomposition-xi}, arising from the standard decomposition of the symmetric two-tensor $\xi$. The same conclusion of the main theorem holds when  $(Y,h_0) = (S^{n-1}, g_{S^{n-1}})$ and $f - G$ is not a first spherical harmonic of $S^{n-1}$.

	(b) It would be worth understanding whether, for unstable Einstein metrics on the conical links, there might be cases in which the quadratic form in~\eqref{eq:2nd} is not negative-definite. This might possibly lead to other examples for which $\mathcal{Y}(M,[\bar{g}])$ is not attained, as in \cite{Ak-Mon}.

 (c) It would be interesting to perform expansions of the Yamabe quotient for other types of metric singularities, and whether there are general criteria to control the Yamabe constant from above to ensure compactness of minimizing sequences. 
\end{remark}

\

The plan of the paper is the following. In Section \ref{s:prel} we collect some preliminary notation and results. We need in particular the standard decomposition 
of symmetric two-tensors, variations of interesting geometric quantities as well as 
some conformal aspects. In Section \ref{s:exp-I} we display the second 
variation of the Yamabe quotient at a conical metric and in particular the role of the 
first-order correction $\Psi$ to the standard bubble. 
In Section  \ref{s:exp-II} we derive rigorously the second order expansion of the Yamabe quotient, 
introducing some cut-off functions that allow to localize the estimates near 
a given conical point. The four-dimensional case is treated separately since it 
leads to logarithmically diverging quantities, and more care is needed.

\medspace \medspace \medspace 

\begin{center}
	{\bf Acknowledgments} 
\end{center}

\noindent 
A.M. is supported by the project {\em Geometric problems with loss of compactness} from Scuola Normale Superiore
and by the PRIN Project 2022AKNSE4 {\em Variational and Analytical aspects of Geometric PDEs}.
 Both authors are members of GNAMPA, as part of INdAM. We are grateful to K. Akutagawa, K. Kr\"oncke and J. Viaclovsky for pointing out to us some useful references.

%% file: Section2.tex
\section{Some preliminary facts}\label{s:prel}

In this section we collect some notation, basic material 
on tensor decompositions, variations of geometric quantities with 
respect to the metric, functional-analytic properties and a 
characterization of the Yamabe constant for some purely conical manifolds. 
We also introduce our conventions on curvatures and differential operators, 
which might differ from some of the existing literature. 

\smallskip

Let $(M,g)$ be a Riemannian manifold. We define the Riemann curvature tensor $\mathrm{Riem}(g)$ (as a $(3,1)$-tensor) with the following sign convention:
\begin{equation}
    \label{eq:Riemann-tensor}
    \mathrm{Riem}(g)(X,Y,Z) = \nabla_X \nabla_Y Z - \nabla_Y \nabla_X Z - \nabla_{[X,Y]} Z.
\end{equation}
Moreover, we define the Ricci tensor $\mathrm{Ric}(g)$ as the trace of $\mathrm{Riem}(g)$ with respect to the first argument, and the scalar curvature $R_g$ as the trace of $\mathrm{Ric}(g)$ with respect to the metric $g$ itself. In the sequel, we will omit the dependence on the metric $g$ if it is clear from the context.

In local coordinates the previous quantities have the well known expressions:
\begin{gather}
    \label{eq:coord-riemm}
    \mathrm{Riem}\indices{_{ijk}^l}
    =
    \frac{\partial \mathrm{\Gamma}_{jk}^l}{\partial x_i}
    -
    \frac{\partial \mathrm{\Gamma}_{ik}^l}{\partial x_j}
    +
    \mathrm{\Gamma}_{jk}^m \mathrm{\Gamma}_{im}^l
    -
    \mathrm{\Gamma}_{ik}^m \mathrm{\Gamma}_{jm}^l,
    \quad
    \mathrm{\Gamma}_{ij}^k
    =
    \frac{1}{2} g^{k l} \bigg(
    \frac{\partial g_{jl}}{\partial x_i} + \frac{\partial g_{li}}{\partial x_j}
    - \frac{\partial g_{ij}}{\partial x_l} \bigg) \\[1ex]
    \label{eq:coord-ricci-scal}
    \mathrm{Ric}_{ij} = \mathrm{Riem}\indices{_{kij}^k}, 
    \quad
    R_g = g^{ij} \mathrm{Ric}_{ij}.
\end{gather}

We also need the definition of the divergence of a tensor field. If $T$ is an $(r,s)$-tensor field, then the divergence $\delta T$ of $T$ is the $(r,s-1)$-tensor field obtained as the trace of $\nabla T$ on the first two covariant indices with respect to the metric $g$. 

In local coordinates this amounts to
\begin{equation}
    \label{eq:divergence-T}
    (\delta T)_{j_1 \dots j_{s-1}}^{i_1 \dots i_r} = g^{\alpha \beta} (\nabla T)_{\alpha \beta j_1 \dots j_{s-1}}^{i_1 \dots i_r}.
\end{equation}
Moreover, we define the rough Laplacian $\Delta T$ of $T$ as the $(r,s)$-tensor field given by
\begin{equation}
    \label{eq:Laplacian-tensor}
    \Delta T = \delta (\nabla T).
\end{equation}
In particular, for a function $f$ we have the following formula
\begin{equation}
    \label{eq:Laplacian-coord}
    \Delta f = \delta (d f) = \frac{1}{\sqrt{\mathrm{det}(g)}} \frac{\partial}{\partial x_i} \bigg( \frac{\partial f}{\partial x_j} g^{ij} \sqrt{\mathrm{det}(g)} \bigg).
\end{equation}

The metric $g$ induces naturally a scalar product on the space of $(r,s)$-tensors, that from now on we will denote by the symbol $\langle T,S \rangle_g$, or simply by $\langle T,S \rangle$ when the metric $g$ is clear from the context. Additionally, the metric $g$ also induces a measure $\mu_g$ on $M$ (in coordinates it takes the form of $\sqrt{\mathrm{det}(g)}$ times the Lebesgue measure). The key property of the divergence operator introduced above is that it is the formal adjoint, with respect to the scalar product of $L^2(M, d\mu_g)$, of the covariant derivative. More precisely, if $M$ is compact and without boundary, $T$ is an $(r,s)$-tensor field and $S$ is an $(r,s+1)$-tensor field, then
\begin{equation}
    \label{eq:by-parts-div}
    \int_M \langle \nabla T, S \rangle \, d\mu_g = - \int_M \langle T, \delta S \rangle \, d \mu_g.
\end{equation}

We now consider the case of symmetric $(0,2)$-tensors more in detail. If we restrict the domain of the divergence operator to these tensors, then its range is contained in the space of $1$-forms. We introduce the operator $\delta^*$ as the formal adjoint of $\delta$, with respect to $L^2(M, d\mu_g)$, between these two spaces. More precisely, given a $1$-form $\omega$ we have
\begin{equation}
    \label{eq:delta-star}
    \delta^* \omega (X,Y) = - \frac{1}{2} \big( \nabla \omega (X,Y) + \nabla \omega(Y,X) \big).
\end{equation}
For the sake of clarity we remark here that although both $\delta^*$ and $-\nabla$ are adjoints of $\delta$, they do not coincide in the space of $1$-forms because the former takes values in the space of symmetric $(0,2)$-tensors, while the latter takes values in the space of \emph{all} tensors of type $(0,2)$.

We recall also the definition of the Einstein operator $\Delta_E$ on symmetric $(0,2)$-tensors. In particular, if $\xi$ is a symmetric $(0,2)$-tensor we define  $\Delta_E \xi$ by the formula
    \begin{equation}
        \label{eq:einstein-operator}
        \Delta_E \xi = \Delta \xi + 2 A_g(\xi),
    \end{equation}
where $A_g(\xi)$ is the symmetric $(0,2)$-tensor defined by
    \begin{equation}
        \label{eq:AA-operator}
        A_g(\xi)(X,Y) := \mathrm{tr} \big( (Z,W) \mapsto \xi(\mathrm{Riem}(g)(Z,X,Y), W) \big),
    \end{equation}
which in local coordinates reduces to
    \begin{equation}
        \label{eq:AA-operator-coord}
        A_g(\xi)_{ij} = g^{\alpha \beta} \xi_{l \beta} \mathrm{Riem}\indices{_{\alpha ij}^l}.
    \end{equation}
    
\subsection{Variations of geometric quantities and tensor decompositions}\label{ss:variations}

We now collect some useful formulas for the variations of certain geometric quantities depending on a smooth family of metrics $h(s)$ on a fixed closed manifold $Y$. 
We start by specifying our assumptions on the dependence of 
$h(s)$ in $s$. Since $Y$ is compact, one can introduce a $C^2$-norm on the symmetric two-tensor fields $\Gamma(\mathcal{S}^2(Y))$ on $Y$ by taking the supremum of the $C^2$-norms of the  
coefficients for a finite atlas on $Y$. While such quantity would depend on the 
choice of charts, all these norms would be equivalent. We then assume 
that the map $s \mapsto h(s)$ is of class $C^3$ from a closed interval $[0,\delta]$ (or $[0,+\infty)$) to $\Gamma(\mathcal{S}^2(Y))$ endowed with one of these equivalent $C^2$-norms. This condition, which might be possibly slightly relaxed,
will be used in the following expansions and also in the estimates of Lemma \ref{l:errori-precisi}. 

\begin{proposition}
    \label{prop:espansione-operatori-2}
    Let $Y$ be a closed $(n-1)$-manifold and let $h \colon [0,+\infty) \to \mathrm{Met}_2(Y)$ be a function of class $C^3$, where $\mathrm{Met}_2(Y)$ denotes the space of $C^2$-metrics on $Y$. Then the following formulas for the 
    volume form and the scalar curvature of $(Y,h(s))$ hold:
    \begin{align}
        \label{eq:measure-2}
        & \frac{d}{ds} \mu_h = \frac{1}{2} \big( \mathrm{tr}_h h' \big) \mu_h, \\[1.5ex]
        \label{eq:measure-3}
        & \frac{d}{ds} \mathrm{tr}_h h' = - \abs{h'}^2_h + \mathrm{tr}_h h'', \\[1.5ex]
        \label{eq:scal-derivata-prima}
        & \frac{d}{ds} R_h = - \Delta_h \mathrm{tr}_h h' + \delta^2_h h' - \big\langle \mathrm{Ric}(h), h' \big\rangle_h,
    \end{align}
    where in the previous formulas all the quantities are implicitly evaluated at $s$.
    
    Moreover, if $h_0:=h(0)$ is an Einstein metric with Einstein constant equal to $(n-2) \sigma$, and if we set $\xi:=h'(0)$ and $\eta:=h''(0)$, then there exists a $1$-form $\beta=\beta(\xi)$ such that
    \begin{equation} \label{eq:der-scalar-2}
        \frac{d^2}{ds^2} R_h |_{s=0}
        =
        \langle B_{h_0}(\xi) , \xi \rangle
        -
        \abs{\delta \xi}^2
        -
        \frac{1}{2} \abs{d(\mathrm{tr} \xi)}^2
        +
        \frac{3}{2} \langle \delta \xi ,  d (\mathrm{tr} \xi) \rangle
        +
        \mathsf{d} R_{h_0}[\eta]
        +
        \delta \beta,
    \end{equation}
    where we have set
    \begin{align}
    \label{eq:B-op}
    B_{h_0}(\xi) & := \frac{1}{2} \Delta_E \xi + \frac{3}{2} \nabla^2 (\mathrm{tr} \xi) + 2 \delta^*(\delta \xi)
    +
    (n-2) \sigma \xi, \\[1.5ex]
    \label{eq:scalar-differential}
    \mathsf{d} R_{h_0}[\eta] & :=
    - \Delta \mathrm{tr} \hspace{0.02cm} \eta + \delta^2 \eta - \big\langle \mathrm{Ric}(h_0), \eta \big\rangle
    =
    - \Delta \mathrm{tr} \hspace{0.02cm} \eta + \delta^2 \eta - (n-2) \sigma \mathrm{tr} \hspace{0.02cm} \eta,
    \end{align}
    and where all the scalar products, the norms and the differential operators appearing in the previous formulas are implicitly taken with respect to the metric $h_0$.
\end{proposition}

A proof of the first part of the previous proposition can be found in~\cite{Besse}, while a proof of the second part can be found in Lemma~2.2 in~\cite{Da-Kro}. We point out that our sign convention for the Riemann curvature tensor is the same as in~\cite{Da-Kro} but our definition of the divergence operator differs from their definition by a minus sign, and therefore the same goes for the definition of the rough Laplacian and the Einstein operator. We notice also that in the statement of Lemma~2.2 in~\cite{Da-Kro} the Einstein operator does not appear explicitly, where it appears instead the Lichnerowicz Laplacian, but in the Einstein case, as remarked by the authors before the proof, these two operators differ by two times the Einstein constant times the identity operator.

The following proposition becomes a corollary of the previous one, after an integration by parts.

\begin{proposition}
    \label{prop:second-var-EH}
    Let $(Y,h_0)$ be a closed Einstein $(n-1)$-manifold with Einstein constant equal to $(n-2) \sigma$.
    Let $\xi$ be a symmetric $(0,2)$-tensor such that
        \begin{equation} 
            \int_Y \mathrm{tr} \xi \, d\mu_{h_0} = 0.
        \end{equation}
    Then, the second variation of the Einstein-Hilbert functional $\mathcal{R}$ defined by~\eqref{eq:def-EH} at $h_0$ in the direction $\xi$ is equal to
    \begin{multline}
        \label{eq:second-var-EH-simplified}
        \mathcal{R}''(h_0)[\xi] = \mathrm{Vol}_{h_0}(Y)^{-\frac{n-3}{n-1}}
        \int_Y
        \Big( -\frac{1}{2} \abs*{\nabla \xi}^2 
        + \abs*{\delta \xi}^2
        + \frac{1}{2} \abs*{\nabla \mathrm{tr}\xi}^2 \\[1ex]
        - \langle \delta \xi , d(\mathrm{tr}\xi) \rangle
        - \frac{(n-2)\sigma}{2} (\mathrm{tr}\xi)^2
        + \langle A_{h_0}(\xi), \xi \rangle \Big)
        \, d\mu_{h_0},
    \end{multline}
    where in the previous formula $A_{h_0}$ is the operator defined by~\eqref{eq:AA-operator}.
\end{proposition}

We remark here that the set of critical metrics for $\mathcal{R}$ coincides with the set of Einstein metrics on $Y$, therefore if $h_0$ is an Einstein metric, then the second variation of $\mathcal{R}$ is well-defined at $h_0$ and depends only on the direction $\xi$.

At this point, we want to understand better the action on the space $\Gamma(\mathcal{S}^2(Y))$ of symmetric $(0,2)$-tensor fields of some of the differential operators introduced in the previous propositions. To this aim it is useful to recall the following well known decomposition, for a proof see~\cite{Viaclovsky-2}, Lecture~3. 

\begin{lemma}
    \label{lemma:decomposition-tensor}
    Let $(Y,h_0)$ be a closed Riemannian $(n-1)$-manifold and let $ \delta^* (\Omega^1(Y))$ be the set of symmetric $(0,2)$-tensor fields of the form $\delta^* \omega$ for some $\omega \in \Omega^1(Y)$, where $\Omega^1(Y)$ denotes the set of $1$-forms on $Y$. Then, we have
    \begin{equation}
        \label{eq:decomposition-tensor}
        \Gamma(\mathcal{S}^2(Y)) = \{ f \cdot h_0 \} + \delta^* (\Omega^1(Y)) \oplus \{ \delta \xi = \mathrm{tr} \xi = 0 \}.
    \end{equation}
    Moreover, if $h_0$ is an Einstein metric, then the previous decomposition is a direct sum and it is orthogonal with respect to $\mathcal{R}''(h_0)$, unless $(Y,h_0)$ is isometric to the $(n-1)$-dimensional sphere endowed with its standard metric.
\end{lemma}

Notice that $\delta^* (\Omega^1(Y))$ in formula \eqref{eq:decomposition-tensor} coincides with the family of Lie derivatives of $h_0$ along vector fields over $Y$. The set of symmetric $(0,2)$-tensor fields that belong to the third addend in~\eqref{eq:decomposition-tensor} is usually called the set of transverse-traceless tensors or the set of $tt$-tensors. Clearly, the class of transverse-traceless tensors depends on the background metric $h_0$.

Now, if $h_0$ is an Einstein metric and $\eta$ is a $tt$-tensor, then identity~\eqref{eq:second-var-EH-simplified} reduces to
    \[
         \mathcal{R}''(h_0)[\eta] = -\mathrm{Vol}_{h_0}(Y)^{-\frac{n-3}{n-1}}
        \int_Y
        \Big(\frac{1}{2} \abs*{\nabla \eta}^2 
        - \langle A_{h_0}(\eta), \eta \rangle \Big)
        \, d\mu_{h_0}.
    \]
Recall that  $(Y,h_0)$ is said to be a strictly stable Einstein manifold if for any $tt$-tensor $\eta$ one has 
\begin{equation}
    \label{eq:R-parte-tt}
    \mathcal{R}''(h_0)[\eta, \eta] \le 0 \ \text{with equality if and only if $\eta=0$}.
\end{equation}

Moreover, for tensors that are multiples with zero average of the metric (see e.g.~Proposition~2.1 in~\cite{Viaclovsky-2}), there holds 
\begin{equation}
    \label{eq:R-parte-conforme}
    \mathcal{R}''(h_0)[f \cdot h_0,f \cdot h_0] = \frac{(n-2)(n-3)}{2} \mathrm{Vol}_{h_0}(Y)^{-\frac{n-3}{n-1}} \int_Y \Big( \abs{\nabla f}^2 - (n-1) f^2 \Big) \, d \mu_{h_0}.
\end{equation}
Furthermore, given the invariance by diffeomorphisms action, we also have that 
\begin{equation}
    \label{eq:R-parte-Lie}
    \mathcal{R}''(h_0)[\tau, \tau] = 0 \ \text{ for all $\tau \in \delta^* (\Omega^1(Y))$}, 
\end{equation}
see e.g. the proof of Proposition~3.4 in~\cite{Viaclovsky-2}.

Using Hodge's decomposition, see e.g.~Theorem 6.8 in~\cite{Warner}, and 
Bochner's formula for positive-Ricci metrics one find that every 
element $\omega$ of $\Omega^1(Y)$ writes as 
\[
\omega = - d G + \alpha, \ \text{with $G: Y \to \mathbb{R}$ and $\delta \alpha = 0$}. 
\]
Using this information, together with the previous lemma, we deduce that
\begin{equation}
    \label{eq:decomposition-xi}
    \xi = f h_0 + \nabla^2 G + \delta^*\alpha + \eta,
\end{equation}
where $G$ is a smooth function with zero average, $\delta \alpha = 0$, and $\eta$ is a transverse-traceless tensor with respect to $h_0$. We point out here that $\mathrm{tr} \hspace{0.02cm} \delta^* \alpha = - \delta \alpha = 0$.

Next, the following proposition will be useful. 
\begin{proposition}
    \label{prop:div-hess}
    Let $(Y,h_0)$ be a closed Einstein $(n-1)$-manifold with Einstein constant $(n-2) \sigma$. Let $G$ and $\omega$ be respectively a function and a $1$-form on $Y$. Then, one has:
        \begin{equation} \label{eq:one-div}
            \delta (\nabla^2 G) = d \big( \Delta G + (n-2) \sigma G \big), \qquad \delta (\delta^* \omega) = - \frac{1}{2} d (\delta \omega) - \frac{(n-2) \sigma}{2} \omega - \frac{1}{2} \Delta \omega,
        \end{equation}
    and also
        \begin{equation} \label{eq:div-hess2}
            \delta^2 (\nabla^2 G) = \Delta^2 G + (n-2) \sigma \Delta G, \qquad \delta^2 (\delta^* \omega) = - \Delta (\delta \omega) 
            - \frac{(n-2) \sigma}{2} \delta \omega. 
        \end{equation}
\end{proposition}

\begin{proof}
The four identities are a consequence of the following commutation rule for 
$1$-forms:  
\[
(\nabla^2 \omega)_{jpq} = ( \nabla^2 \omega)_{pjq} - 
\mathrm{Riem}\indices{_{pjq}^l} \omega_l, 
\]
which can be also applied to differentials of functions on $Y$.  
\end{proof}

\subsection{Some computations in the conical setting}

We specialize here to conical spaces of the type 
\[
\mathcal{C}(Y) : = [0,+\infty) \times Y, 
\]
with $\{0\} \times Y$ identified to a point, endowed with a metric $g$ as in~\eqref{eq:g-assumption}.  
For such spaces, we compute the expressions of the volume form, the Laplacian and the scalar curvature. 

\begin{proposition}
    \label{prop:espansione-operatori-1}
    Let $Y$ be a closed $(n-1)$-manifold and let $h \colon [0,+\infty) \to \mathrm{Met}_2(Y)$ be a function of class $C^3$, where $\mathrm{Met}_2(Y)$ denotes the space of $C^2$-metrics on $Y$. Let $g$ be the metric on $\mathcal{C}(Y)$ defined as $ds^2 + s^2 h(s)$. Then the following formulas hold:
    \begin{gather}
        \label{eq:measure-1}
        \mu_g = s^{n-1} \mu_h, \\[1ex]
        \label{eq:Laplacian-1}
        \Delta_g = \partial_{ss} + \frac{n-1}{s} \partial_s + \frac{1}{2} \big( \mathrm{tr}_h h' \big) \partial_s + \frac{\Delta_h}{s^2}, \\[1ex]
        \label{eq:curv-scal-1}
        R_g(s,y) = \frac{R_h(y)-(n-1)(n-2)}{s^2} - \frac{n}{s} \mathrm{tr}_{h} h'
        + \frac{3}{4} \abs{h'}^2_{h} - \frac{1}{4} ( \mathrm{tr}_{h} h' )^2 - \mathrm{tr}_{h} h'',
    \end{gather}
    where in the previous formulas all the quantities are implicitly evaluated at $s$.
\end{proposition}

\begin{proof}
    The validity of~\eqref{eq:measure-1} follows immediately from the definition of volume form and the structure of $g$, in particular this means that
        \begin{equation}
            \label{eq:determinanti-rel}
            \sqrt{\mathrm{det}(g)}
            =
            s^{n-1} \sqrt{\mathrm{det}(h)},
        \end{equation}
    where in the previous expression $h$ is implicitly computed at $s$.

    In local coordinates, for any $i,j = 2,\dots,n$, we have
    \begin{equation}
        \label{eq:metric-coeff}
        g_{11} = 1, \quad g_{1j} = 0, \quad g_{ij} = s^2 h_{i-1 j-1}
        \quad \text{and} \quad
        g^{11} = 1, \quad g^{1j} = 0, \quad g^{ij} = s^{-2} h^{i-1 j-1}.
    \end{equation}
    Therefore, from~\eqref{eq:Laplacian-coord},~\eqref{eq:determinanti-rel} and~\eqref{eq:metric-coeff} we obtain
    \begin{align*} 
        \notag
        \Delta_g f & = \frac{1}{s^{n-1} \sqrt{\mathrm{det}(h)}} \frac{\partial}{\partial s} \bigg( \frac{\partial f}{\partial s} s^{n-1} \sqrt{\mathrm{det}(h)} \bigg)
        +
        \frac{1}{s^2} \sum_{i,j \ge 2} 
        \frac{1}{\sqrt{\mathrm{det}(h)}} \frac{\partial}{\partial x_i} \bigg( \frac{\partial f}{\partial x_j} h^{i-1 j-1}\sqrt{\mathrm{det}(h)} \bigg) \\[1.5ex]
        & =
        \frac{1}{s^{n-1}} \frac{\partial}{\partial s} \bigg( \frac{\partial f}{\partial s} s^{n-1} \bigg)
        +
        \frac{1}{\sqrt{\mathrm{det}(h)}} \bigg( \frac{\partial}{\partial s} \sqrt{\mathrm{det}(h)} \bigg) \frac{\partial f}{\partial s}
        + \frac{\Delta_h f}{s^2}.
    \end{align*}
    Combining the previous formula with the identity~\eqref{eq:measure-2} we deduce~\eqref{eq:Laplacian-1}.

    For notation convenience we introduce the shortcut $\mathrm{R}\indices{_{ijk}^l}$ for $\mathrm{Riem}\indices{_{ijk}^l}$.
    The value of the scalar curvature $R_g$ of the metric $g$ at the point $(s,y)$ can be computed in local coordinates using the formula
    \begin{equation}
        \label{eq:scalar-curv-coordinates}
        R_g := g^{ij} \mathrm{R}\indices{_{kij}^k}
        =
        \underbrace{\sum_{k=1}^n g^{11} \mathrm{R}\indices{_{k11}^k}}_{=A}
        +
        \underbrace{\sum_{j=2}^n \sum_{k=1}^n g^{1j}\mathrm{R}\indices{_{k1j}^k}}_{=B}
        +
        \underbrace{\sum_{i=2}^n \sum_{k=1}^n g^{i1} \mathrm{R}\indices{_{ki1}^k}}_{=C}
        +
        \underbrace{\sum_{i,j=2}^n \sum_{k=1}^n g^{ij} \mathrm{R}\indices{_{kij}^k}}_{=D}.
    \end{equation}
    To simplify the previous sum, we need the precise expression of the Christoffel symbols, from~\eqref{eq:coord-riemm} we easily obtain the following identities:
    \begin{align}
        \label{eq:C1}
        \mathrm{\Gamma}_{ij}^k & = 0 \ \text{if at least two indices are equal to one}, \\[1ex]
        \label{eq:C2}
        \mathrm{\Gamma}_{ij}^1 & = - s h_{i-1 j-1} -\frac{s^2}{2} h_{i-1 j-1}' \ \text{for any $i,j=2,\dots,n$}, \\[1ex]
        \label{eq:C3}
        \mathrm{\Gamma}_{1j}^k & = \frac{1}{2} \big( h^{-1} h' \big)_{k-1 j-1} + \frac{1}{s} \delta_{j}^{k} \ \text{for any $j,k=2,\dots,n$}, \\[1ex]
        \label{eq:C4}
        \mathrm{\Gamma}_{ij}^k & = \widetilde{\mathrm{\Gamma}}_{i-1 j-1}^{k-1} \ \text{for any $i,j,k=2,\dots,n$},
    \end{align}
    where $\widetilde{\mathrm{\Gamma}}_{\alpha \beta}^{\gamma}(s,y)$ denotes a generic Christoffel symbol associated to the Levi-Civita connection on $(Y,h(s))$. 
    It is clear from~\eqref{eq:metric-coeff} that $B=C=0$ (where $B$ and $C$ are two of the four addends that appear in the formula~\eqref{eq:scalar-curv-coordinates}). 
    
    We turn our attention to the term $A$. From~\eqref{eq:C1} we deduce that
    \begin{equation}
        \label{eq:A1}
        \mathrm{R}\indices{_{k11}^k}
        :=
        \frac{\partial \mathrm{\Gamma}_{11}^k}{\partial x_k}
        -
        \frac{\partial \mathrm{\Gamma}_{k1}^k}{\partial s}
        +
        \mathrm{\Gamma}_{11}^m \mathrm{\Gamma}_{km}^k
        -
        \mathrm{\Gamma}_{k1}^m \mathrm{\Gamma}_{1m}^k 
        =
        -
        \frac{\partial \mathrm{\Gamma}_{k1}^k}{\partial s}
        -
        \mathrm{\Gamma}_{k1}^m \mathrm{\Gamma}_{1m}^k.
    \end{equation}
    Combining the identity $( h^{-1} )' = - h^{-1} h' h^{-1}$ with the relation~\eqref{eq:C3} we obtain that for any $k \ge 2$ it holds
    \begin{align}
        \notag
        \frac{\partial \mathrm{\Gamma}_{k1}^k}{\partial s}
        & =
        \frac{1}{2} \frac{\partial}{\partial s} \big( h^{-1} h' \big)_{k-1 k-1} - \frac{1}{s^2} \\[1ex]
        \label{eq:A2}
        & =
        \frac{1}{2} \big( h^{-1} h'' \big)_{k-1 k-1} - \frac{1}{2} \big( h^{-1} h' h^{-1} h' \big)_{k-1 k-1} - \frac{1}{s^2}.
    \end{align}
    Moreover, from~\eqref{eq:C1} and \eqref{eq:C3} we also derive that for any $k \ge 2$ it holds
    \begin{align}
        \notag
        \mathrm{\Gamma}_{k1}^m \mathrm{\Gamma}_{1m}^k
        = & 
        \frac{1}{s^2} \sum_{m=2}^n \delta_{k}^{m} \delta_{m}^{k}
        +
        \frac{1}{4} \sum_{m=2}^n \big( h^{-1} h' \big)_{m-1 k-1} \big( h^{-1} h' \big)_{k-1 m-1} \\[1ex]
        \notag
        & + 
        \frac{1}{2s} \sum_{m=2}^n \delta_{k}^{m} \big( h^{-1} h' \big)_{k-1 m-1}
        +
        \frac{1}{2s} \sum_{m=2}^n \delta_{m}^{k} \big( h^{-1} h' \big)_{m-1 k-1} \\[1ex]
        \label{eq:C5}
        = &
        \frac{1}{s^2} + \frac{1}{4} \big( h^{-1} h' h^{-1} h' \big)_{k-1 k-1} + \frac{1}{s} \big( h^{-1} h' \big)_{k-1 k-1}.
    \end{align}
    Putting together~\eqref{eq:C1},~\eqref{eq:A1},~\eqref{eq:A2} and~\eqref{eq:C5} we obtain
    \begin{align}
        \notag
        A & = \sum_{\alpha =1}^{n-1} \frac{1}{4} \big( h^{-1} h' h^{-1} h' \big)_{\alpha \alpha} - \frac{1}{s} \big( h^{-1} h' \big)_{\alpha \alpha} - \frac{1}{2} \big( h^{-1} h'' \big)_{\alpha \alpha} \\[1ex]
        \label{eq:contributo-A}
        & =
        \frac{1}{4} \abs{h'}_h^2 - \frac{1}{s} \mathrm{tr}_h h' - \frac{1}{2} \mathrm{tr}_h h''.
    \end{align}
    We are left with the computation of the term $D$. We split $D$ in the sum of two terms,
    \begin{equation}
        \label{eq:splitting-D}
        D = \underbrace{\sum_{i,j=2}^n g^{ij} \mathrm{R}\indices{_{1ij}^1}}_{=D1}
        +
        \underbrace{\sum_{i,j,k=2}^n g^{ij} \mathrm{R}\indices{_{kij}^k}}_{=D2},
    \end{equation}
    and we start with the contribution originated from $D1$. Using~\eqref{eq:C1} once more we notice that
    \begin{equation}
        \label{eq:D1}
        \mathrm{R}\indices{_{1ij}^1}
        :=
        \frac{\partial \mathrm{\Gamma}_{ij}^1}{\partial s}
        -
        \frac{\partial \mathrm{\Gamma}_{1j}^1}{\partial x_i}
        +
        \mathrm{\Gamma}_{ij}^m \mathrm{\Gamma}_{1m}^1
        -
        \mathrm{\Gamma}_{1j}^m \mathrm{\Gamma}_{im}^1 
        =
        \frac{\partial \mathrm{\Gamma}_{ij}^1}{\partial s}
        -
        \mathrm{\Gamma}_{1j}^m \mathrm{\Gamma}_{im}^1.
    \end{equation}
    From~\eqref{eq:C2} we have
    \begin{equation}
        \label{eq:D2}
       \frac{\partial \mathrm{\Gamma}_{ij}^1}{\partial s} = - h_{i-1 j-1} - 2 s h_{i-1 j-1}' - \frac{s^2}{2} h_{i-1 j-1}''.
    \end{equation}
    Moreover, from~\eqref{eq:C1},~\eqref{eq:C2} and~\eqref{eq:C3} we also derive that
    \begin{align}
        \notag
        \mathrm{\Gamma}_{1j}^m \mathrm{\Gamma}_{im}^1 
        = &
        - \sum_{m=2}^n \delta_{j}^{m} h_{i-1 m-1}
        - \frac{s}{2} \sum_{m=2}^n \delta_{j}^{m} h_{i-1 m-1}'
        \\[1ex]
        \notag
        & - \frac{s}{2} \sum_{m=2}^n \big( h^{-1} h' \big)_{m-1 j-1} h_{i-1 m-1}
        - \frac{s^2}{4} \sum_{m=2}^n \big( h^{-1} h' \big)_{m-1 j-1} h_{i-1 m-1}' \\[1ex]
        \label{eq:C6}
        = &
        - h_{i-1 j-1} - s h_{i-1 j-1}' - \frac{s^2}{4} \big( h' h^{-1} h' \big)_{i-1 j-1}.
    \end{align}
    Therefore, putting together~\eqref{eq:metric-coeff},~\eqref{eq:D1},~\eqref{eq:D2} and~\eqref{eq:C6} we obtain
    \begin{align}
        \notag
        D1 & = \frac{1}{s^2} \sum_{\alpha, \beta = 1}^{n-1} h^{\alpha \beta} \Big( - s h_{\alpha \beta}' - \frac{s^2}{2} h_{\alpha \beta}'' + \frac{s^2}{4} \big( h' h^{-1} h' \big)_{\alpha \beta} \Big) \\[1ex]
        \notag
        & =
        - \frac{1}{s^2} \sum_{\alpha = 1}^{n-1} s \big( h^{-1} h' \big)_{\alpha \alpha} + \frac{s^2}{2} \big( h^{-1} h'' \big)_{\alpha \alpha} - \frac{s^2}{4} \big(h^{-1} h' h^{-1} h' \big)_{\alpha \alpha} \\[1ex]
        & =
        \label{eq:contributo-D1}
        - \frac{1}{s} \mathrm{tr}_h h' - \frac{1}{2} \mathrm{tr}_h h'' + \frac{1}{4} \abs{h'}_h^2. 
    \end{align}

    It  only remains the contribution of the term $D2$. We point out that if $i,j,k$ and $l$ are greater or equal than two, then we can write $\mathrm{R}\indices{_{ijk}^l}$ in terms of the coefficients $\widetilde{\mathrm{R}}\indices{_{\alpha \beta \gamma}^{\kappa}}$ of the Riemann curvature tensor on $(Y,h(s))$. More precisely
    \begin{align}
        \notag
        \mathrm{R}\indices{_{ijk}^l}
        & =
        \Bigg( \frac{\widetilde{\mathrm{\Gamma}}_{j-1 k-1 }^{l-1}}{x_{i-1}}
        -
        \frac{\widetilde{\mathrm{\Gamma}}_{i-1 k-1}^{l-1}}{x_{j-1}}
        +
        \sum_{m = 2}^n \Big(
        \widetilde{\mathrm{\Gamma}}_{j-1 k-1}^{m-1}\widetilde{\mathrm{\Gamma}}_{i-1 m-1}^{l-1}
        -
        \widetilde{\mathrm{\Gamma}}_{i-1 k-1}^{m-1} \widetilde{\mathrm{\Gamma}}_{j-1 m-1}^{l-1}
        \Big)
        \Bigg)
        +
        \mathrm{\Gamma}_{jk}^1 \mathrm{\Gamma}_{i1}^l
        -
        \mathrm{\Gamma}_{ik}^1 \mathrm{\Gamma}_{j1}^l \\[1.5ex]
        \notag
        & =
        \widetilde{\mathrm{R}}\indices{_{i-1 j-1 k-1}^{l-1}}
        +
        \mathrm{\Gamma}_{jk}^1 \mathrm{\Gamma}_{i1}^l
        -
        \mathrm{\Gamma}_{ik}^1 \mathrm{\Gamma}_{j1}^l.
    \end{align}
    In particular
    \begin{equation}
        \label{eq:decomp-Rbis}
        \mathrm{R}\indices{_{kij}^k} = \widetilde{\mathrm{R}}\indices{_{k-1 i-1 j-1}^{k-1}}
        +
        \mathrm{\Gamma}_{ij}^1 \mathrm{\Gamma}_{k1}^k
        -
        \mathrm{\Gamma}_{kj}^1 \mathrm{\Gamma}_{i1}^k.
    \end{equation}
    Moreover, using again~\eqref{eq:C2} and~\eqref{eq:C3} we have the following identities
    \begin{equation*}
        \mathrm{\Gamma}_{ij}^1 \mathrm{\Gamma}_{k1}^k
        =
        - h_{i-1 j-1} - \frac{s}{2} h_{i-1 j-1}' - \frac{s}{2} h_{i-1 j-1} \big( h^{-1} h' \big)_{k-1 k-1} - \frac{s^2}{4} h_{i-1 j-1}' \big( h^{-1} h' \big)_{k-1 k-1} 
    \end{equation*}
    and
    \begin{equation*}
        \mathrm{\Gamma}_{kj}^1 \mathrm{\Gamma}_{i1}^k
        =
        - h_{k-1 j-1} \delta_{i}^{k}
        - \frac{s}{2} h_{k-1 j-1}' \delta_{i}^{k}
        - \frac{s}{2} h_{k-1 j-1} \big( h^{-1} h' \big)_{k-1 i-1}
        - \frac{s^2}{4} h_{k-1 j-1}' \big( h^{-1} h' \big)_{k-1 i-1}.
    \end{equation*}
    From the previous relations and~\eqref{eq:metric-coeff} we obtain
    \begin{align}
        \notag
        \sum_{i,j,k = 2}^n g^{ij}
        \mathrm{\Gamma}_{ij}^1 \mathrm{\Gamma}_{k1}^k
         = &
        - \frac{1}{s^2} \sum_{\alpha, \beta, \kappa = 1}^{n-1} h^{\alpha \beta} \big(
        h_{\alpha \beta} + \frac{s}{2} h_{\alpha \beta}' \big) \\[1ex]
        \notag
        & - 
        \frac{1}{s^2} \sum_{\alpha, \beta, \kappa = 1}^{n-1} h^{\alpha \beta} \Big(
        \frac{s}{2} h_{\alpha \beta} \big( h^{-1} h' \big)_{\kappa \kappa}
        +
        \frac{s^2}{4} h_{\alpha \beta}' \big( h^{-1} h' \big)_{\kappa \kappa}
        \Big) \\[1ex]
        \label{eq:C7}
         = &
        - \frac{(n-1)^2}{s^2} - \frac{n-1}{s} \mathrm{tr}_h h' - \frac{1}{4} (\mathrm{tr}_h h')^2, 
    \end{align}
    and
    \begin{align}
        \notag
        \sum_{i,j,k = 2}^n g^{ij} 
        \mathrm{\Gamma}_{kj}^1 \mathrm{\Gamma}_{i1}^k 
        = &
        - \frac{1}{s^2} \sum_{\alpha, \beta, \kappa = 1}^{n-1} h^{\alpha \beta} \big(
        h_{\kappa \beta} \delta_{\alpha}^{\kappa}
        +
        \frac{s}{2} h_{\kappa \beta}' \delta_{\alpha}^{\kappa}
        \big) \\[1ex]
        \notag
        & -
        \frac{1}{s^2} \sum_{\alpha, \beta, \kappa = 1}^{n-1} h^{\alpha \beta} \Big(
        \frac{s}{2} h_{\kappa \beta} \big( h^{-1} h' \big)_{\kappa \alpha}
        +
        \frac{s^2}{4} h_{\kappa \beta}' \big( h^{-1} h' \big)_{\kappa \alpha}
        \Big) \\[1ex]
        \label{eq:C8}
        = &
        - \frac{n-1}{s^2} - \frac{1}{s} \mathrm{tr}_h h' 
        - \frac{1}{4} \abs{h'}_{h}^2.
    \end{align}
    Combining~\eqref{eq:metric-coeff},~\eqref{eq:decomp-Rbis},~\eqref{eq:C7} and~\eqref{eq:C8} we deduce
    \begin{align}
        \notag
        D2 & = \frac{1}{s^2}
        \sum_{\alpha, \beta, \kappa = 1}^{n-1}
        h^{\alpha \beta} \widetilde{\mathrm{R}}\indices{_{\kappa \alpha \beta}^{\kappa}}
        +
        \sum_{i,j,k = 2}^n g^{ij} \big(
        \mathrm{\Gamma}_{ij}^1 \mathrm{\Gamma}_{k1}^k
        -
        \mathrm{\Gamma}_{kj}^1 \mathrm{\Gamma}_{i1}^k \big) \\[1ex]
        \label{eq:contributo-D2}
        & =
        \frac{R_{h(s)}(y) - (n-1)(n-2)}{s^2} - \frac{n-2}{s} \mathrm{tr}_h h' - \frac{1}{4} (\mathrm{tr}_h h')^2 + \frac{1}{4} \abs{h'}_{h}^2.
    \end{align}
    Finally, the conclusion follows summing up the identities~\eqref{eq:contributo-A},~\eqref{eq:contributo-D1} and~\eqref{eq:contributo-D2}.
\end{proof}

\begin{remark}
    For a metric $g$ of the form $ds^2 + f(s,y)^2 h_0$ with a smooth function $f$ on $(0,+\infty) \times Y$, similar computations are performed also in the paper~\cite{Leung}, Section~4. 
\end{remark}

\subsection{Functional-analytic aspects and Yamabe quotient of cones} \label{ss:FAYQ}

Weak solutions of 
the Yamabe equation \eqref{eq:Y} can be found in the Sobolev space 
\[
    W^{1,2}(M,\bar{g}) := \Big\{ \text{closure of $\mathrm{Lip}(M)$ 
    with respect to $\norm{\nabla \, \cdot \, }_{L^2(d\mu_{\bar{g}})} + \norm{ \, \cdot \, }_{L^2(d\mu_{\bar{g}})}$} \Big \}. 
\]
Now, if $(M,\bar{g})$ is a compact singular Riemannian manifold with a finite number of conical points and for any of them the metric $\bar{g}$ satisfies~$(H_P)$, then one has the standard Sobolev embeddings of~$W^{1,2}(M,\bar{g})$ into~$L^p(M,d\mu_{\bar{g}})$, for $1 \leq p \leq 2n/(n-2)$, which are compact 
when~$p < 2n/(n-2)$, see e.g.~Propositions~1.6 and~2.2 in~\cite{ACM} or Proposition~1.3 in \cite{CLV}.

Under our assumptions, implying in particular that the scalar curvature 
is of class $L^p(M,d\mu_{\bar{g}})$ for any~$p \in (n/2, n)$, the Yamabe 
constant $\mathcal{Y}(M,[\bar{g}])$ is positive. For more details, we again refer the reader to Section~1 in~\cite{ACM}, that treats way more general metric structures, as discussed in 
Remark \ref{rem:intro}~$(b)$.

For the class of manifolds as in Theorem \ref{t:main}, the 
regularity of conformal factors satisfying the Yamabe equation is well
understood. We recall the following result, which consists of 
Proposition 3.1 in \cite{ACM}
(see also \cite{Ak-Bo}). 

\begin{proposition} \label{p:holder}
Let $(M,\bar{g})$ be as in Theorem~\ref{t:main}, and let $u$ be a positive solution of the Yamabe equation~\eqref{eq:Y} on the regular part of~$M$ such that~$u(x) \leq C s^{-\frac{n-2}{2}+\delta}$ for some positive constants $C$ and $\delta$, where~$s = \min_i d_{\bar{g}}(x,P_i)$. Then, locally near each conical point~$P_i$ the function~$u$ is continuous with
$\abs{u(x) - u (P_i)} = O(s^{\alpha_i})$, where~$\alpha_i$ is determined by the first non-trivial eigenvalue~$\lambda_{1,i}$ of the following operator:
    \[
        - \Delta_{h_{i,0}} + \frac{(n-2)^2}{4} \ \text{on the link $(Y_i,h_{i,0})$ at $P_i$}.
    \]
More precisely, we have~$\alpha_i = \sqrt{\lambda_{1,i}} - \frac{n-2}{2}$. 
\end{proposition}

\begin{remark} 
If a non-negative function $u \in W^{1,2}(M,\bar{g})$ realizes $\mathcal{Y}(M,[\bar{g}])$, then it 
satisfies the assumptions of Proposition \ref{p:holder}, by 
Theorem 1.12 in \cite{ACM}. It would indeed be  possible to show the above upper bound for every weak $W^{1,2}$ solution of \eqref{eq:Y}, for example 
 by cylindrically blowing-up the metric at the conical points and proving exponential decay along the ends. 
\end{remark}

\smallskip

For understanding the possible blow-up at conical points, it is useful to 
compare both the geometry and the Yamabe constant of the compact 
manifold to that of the limiting cones at singular points that are obtained by blowing-up the metric. For a closed Einstein $(n-1)$-manifold $(Y,h_0)$ with  Einstein constant $(n-2)$, we consider the {\em cone over} $Y$ 
\[
(\mathcal{C}(Y), g_0) := ([0,+\infty) \times Y , ds^2 + s^2 h_0).
\]
For such a cone we introduce the functional space 
\[
    \mathscr{D}^{1,2}_{g_0} = \Big\{ \text{closure of $C^\infty_c\big((0,\infty) \times Y\big)$ 
    with respect to $\norm{\nabla \, \cdot \,}_{L^2(d\mu_{g_0})}$} \Big\}.
\]

\begin{remark} \label{rmk:fourier-D}
    Let $\{ u_k \}_{k \in \mathbb{N}}$ be an orthonormal basis of $L^2(Y,d\mu_{h_0})$ consisting in eigenfunctions for $-\Delta_{h_0}$. Let $\Psi \in \mathscr{D}^{1,2}_{g_0}$, then there exists a sequence of functions $\{ \psi_{\lambda_k} \}_{k \in \mathbb{N}}$ such that
    \[
        \Psi(r,y) = \sum_{k = 0}^{+\infty} \psi_{\lambda_k}(r) u_k(y) \ \text{with} \ \psi_{\lambda_k}(r) = \int_Y \Psi(r,y) u_k(y) \, d \mu_{h_0} \ \text{for almost every $r \in (0,+\infty)$}.
    \]
\end{remark}

By formula~(2.4) in~\cite{ACM}, under the assumptions of Theorem 
\ref{t:main}, given a 
conical point $P$  the constant $\mathcal{Y}_P$ coincides with the Yamabe constant of $(\mathcal{C}(Y), g_0)$, which is indeed explicitly known by the 
following result of Petean (Corollary~1.3 in~\cite{Petean}), see also \cite{Mond} for an extension and \cite{Nob-Vio} for related 
results in the setting of RCD spaces.  

\begin{proposition} \label{p:Petean}
    Let $(Y,h_0)$ be a closed Einstein $(n-1)$-manifold with Einstein constant equal to $n-2$, and let $(\mathcal{C}(Y), g_0)$ be the metric cone over $Y$. Then 
    \[
        \mathcal{Y}(\mathcal{C}(Y),g_0) = \bigg( \frac{\mathrm{Vol}_{h_0} (Y)}{\mathrm{Vol}_{g_{S^{n-1}}} (S^{n-1})} \bigg)^{\frac{2}{n}} \mathcal{Y}(S^n,g_{S^n}),
    \]
    and the above constant is attained by the function $U(r,y) = U(r)$ given by 
    \begin{equation} \label{eq:bolla-std}
        U(r) := c_Y \bigg( \frac{1}{1+r^2} \bigg)^{\frac{n-2}{2}} \! ,
        \quad \text{with} \quad
        c_Y := \bigg( \frac{4n(n-1)}{\mathcal{Y}(\mathcal{C}(Y),g_0)} \bigg)^{\frac{n-2}{4}},
    \end{equation}
    where the constant $c_Y$ has been chosen so that the $L^{\frac{2n}{n-2}}(d \mu_{g_0})$ norm of $U$ is equal to one.
 \end{proposition}

As a consequence of the previous proposition and the choice of the constant $c_Y$ the function~$U$ defined in~\eqref{eq:bolla-std} solves the equation 
    \begin{equation} \label{eq:lagrange-mult-U}
        - a \Delta_{g_0} U = \mathcal{Y}_P U^{\frac{n+2}{n-2}}.
    \end{equation}

Moreover, for every $\eps >0$ also the function $U_\eps(r):= \eps^{-\frac{n-2}{2}} U(r/\eps)$ solves the same equation. Therefore, if we introduce the function
    \begin{equation} \label{eq:V-function}
        V(r):=\frac{\partial}{\partial \eps} U_{\eps}(r) |_{\eps = 1} = \frac{n-2}{2} U(r) + r U'(r),
    \end{equation}
then after taking the derivative in $\eps$ on both sides of the equation~\eqref{eq:lagrange-mult-U} we deduce that $\widetilde{\lp} V = 0$, where
    \begin{equation} \label{eq:V-ker-L}
        \widetilde{\lp} := -a \Delta_{g_0} - \mathcal{Y}_P \frac{n+2}{n-2} U^{\frac{4}{n-2}}.
    \end{equation}
We will refer to $\widetilde{\lp}$ as the linearized operator obtained from the conformal Laplacian $\lp_{g_0}$. 

We consider for a moment the special case $(Y,h_0) = (S^{n-1}, g_{S^{n-1}})$. In this situation, if $\theta \colon S^{n-1} \to \mathbb{R}$ is a first spherical harmonic, that is a function such that $-\Delta_{S^{n-1}} \theta = (n-1) \theta$, then the function $U'(r) \theta(y)$ is a solution of equation~\eqref{eq:V-ker-L}. Indeed, in this case the metric cone $\mathcal{C}(Y)$ is isometric to $\mathbb{R}^n$ endowed with its standard metric and for every $z \in \mathbb{R}^n$ the function $U_z(x):=U(\abs{x-z})$ is again a solution of~\eqref{eq:lagrange-mult-U}. Thus, taking the partial derivative in the direction $z_i$ on both sides of the equation (and evaluating it on $z=0$) we conclude that the function $U'(r) x_i$ is in the kernel of $\widetilde{\lp}$, for every $i=1,\dots,n$. On the other hand, it is well known that the functions $x_i$ generate the eigenspace of $-\Delta_{g_{S^{n-1}}}$ relative to the first non zero eigenvalue $\lambda_1(-\Delta_{g_{S^{n-1}}}) = n-1$. 

Moreover, it has been proved (see for example Lemma~A1 in the Appendix of~\cite{Bianchi-Egnell} or Lemma~5.2 in~\cite{Ambrosetti-Malchiodi-1}) that all the solutions in $\mathscr{D}^{1,2}_{g_0}$ of  equation~\eqref{eq:V-ker-L}  are obtained as a linear combination of the function $V$ defined by~\eqref{eq:V-function} and a function of the type $U' \theta$, where $\theta$ is a first spherical harmonic. 

Now, we turn back to the general case of a closed Einstein $(n-1)$-manifold $(Y,h_0)$ with Einstein constant $(n-2)$. If $(Y,h_0)$ is not isometric to the standard sphere, then the Lichnerowicz-Obata theorem (see e.g. Theorem 2.3 in \cite{Viaclovsky-2}) tells us that the first non zero eigenvalue of the Laplacian $-\Delta_{h_0}$ is strictly greater than $n-1$, which corresponds to the first non zero eigenvalue of the Laplacian on the standard sphere. Combining this observation with a careful reading of the proof in~\cite{Bianchi-Egnell} or in~\cite{Ambrosetti-Malchiodi-1} it is possible to characterize the kernel (in $\mathscr{D}^{1,2}_{g_0}$) of the operator $\widetilde{\lp}$ also in this case. More precisely, the following lemma holds.
    
    \begin{lemma} \label{lemma:kernel-linearized}
        Let $(Y,h_0)$ be a closed Einstein $(n-1)$-manifold with Einstein constant $n-2$, and let $(\mathcal{C}(Y), g_0)$ be the metric cone over $Y$. Assume that $(Y,h_0)$ is not isometric to the standard sphere. Then, any solution in $\mathscr{D}^{1,2}_{g_0}$ of equation~\eqref{eq:V-ker-L} is a constant multiple of the function $V$ defined by~\eqref{eq:V-function}.
    \end{lemma}

    \begin{remark} \label{rmk:U'-w-eing}
        Notice that on $(Y,h_0) = (S^{n-1}, g_{S^{n-1}})$, one has $\widetilde{\lp}(U' x_i) = 0$. This implies that
            \begin{equation} \label{eq:U'-first-we}
                \widetilde{\lp}(U') = -a(n-1)\frac{U'}{r^2} \ \text{for any link $(Y,h_0)$}.
            \end{equation}
        Indeed, the action of $\widetilde{\lp}$ on radial functions does not depend on the metric on the link.
    \end{remark}

We conclude the section stating some useful integral identities (which can be proved by integrating by parts) involving the function $U$.

    \begin{proposition} \label{prop:int-iden}
        Let $U$ be the function defined by~\eqref{eq:bolla-std}. Then, for $n \ge 5$ the following identities hold and all the integrals are finite.
        \begin{align}
            \label{eq:relazioni-integrali-bolla-1}
            \int_0^{+\infty} U(r)^{\frac{2n}{n-2}} r^{n+1} \, d r & = \frac{n^2(n-4)}{\mathcal{Y}_P(n-2)} \int_0^{+\infty} U(r)^2 r^{n-1} \, d r, \\[1.5ex]
            \label{eq:relazioni-integrali-bolla-2}
            \int_0^{+\infty} U(r) U'(r) r^n \, d r & = - \frac{n}{2} \int_0^{+\infty} U(r)^2 r^{n-1} \, d r, \\[1.5ex]
            \label{eq:relazioni-integrali-bolla-3}
            \int_0^{+\infty} \big( U'(r) \big)^2 r^{n+1} \, dr & = \frac{n(n^2-4)}{4(n-1)} \int_0^{+\infty} U(r)^2 r^{n-1} \, d r.
        \end{align}
    \end{proposition}

%% file: Section3.tex
\section{Formal expansions of the Yamabe quotient} \label{s:exp-I}

Let us recall that on a metric cone $(\mathcal{C}(Y),g_0)$ over a link $Y$ as in Theorem \ref{t:main}, 
the Yamabe constant is realized by the radial function $U$, see 
Proposition \ref{p:Petean}. We also recall that 
\[
\mathcal{Y}(\mathcal{C}(Y), g_0) = \mathcal{Y}_P, 
\]
where $\mathcal{Y}_P$ is given by \eqref{eq:YYPP}.

In this section we will consider on $\mathcal{C}(Y)$ a metric $g$ as in~\eqref{eq:g-assumption} and test functions of the form 
\begin{equation}
    \label{eq:def-test-functions}
    u_\eps(s,y) = \eps^{-\frac{n-2}{2}} \big( U(s/\eps) + \eps \Psi (s/\eps,y) \big),
\end{equation}
with $\Psi$ smooth on the regular part of the cone, 
and expand the quotient $Q_g$ on $u_\eps$ formally in powers of 
$\eps$. We will choose $\Psi$ suitably, see Lemma \ref{l:Psi}, 
in order to obtain formally a negative and leading coefficient 
of $\eps^2$ in the expansion. Rigorous estimates will be 
obtained in the next section starting from the formal ones derived here. 

More precisely, we want to study the asymptotic behaviour for small $\eps$ of the quantity
\begin{equation}
    \label{eq:Yamabe-quotient-ev}
    I(\eps):=
    \frac{\int_{\mathcal{C}(Y)} u_\eps \lp_g u_\eps \, d\mu_g}{ \Big( \int_{\mathcal{C}(Y)} u_{\eps}^{\frac{2n}{n-2}} \, d\mu_g \Big)^{\frac{n-2}{n}}},
\end{equation}
where $\lp_g:=-a \Delta_g + R_g$ denotes the conformal Laplacian operator associated to the metric $g$. 
For our later choice of test functions, which will allow  an integration by parts, $I(\eps)$ will coincide with the Yamabe quotient on those.

We consider the change of variable given by $(s,y)=F_\eps(r,y)=(\eps r, y)$, in this way the pull-back metric $F_\eps^{*}g=\eps^2 g_\eps$ becomes, after dividing by $\eps^2$, $g_\eps:=dr^2+r^2h(\eps r)$. Moreover, using the  covariance of the conformal Laplacian \eqref{eq:conf-lap} we get $\lp_{g_\eps} \big( \eps^{\frac{n-2}{2}} (u_\eps \circ F_\eps) \big) = \eps^{\frac{n+2}{2}} (\lp_g u_\eps) \circ F_\eps$. We introduce the notation 
\begin{equation}
    \label{eq:w-eps}
    w_\eps(r,y):=\eps^{\frac{n-2}{2}} u_\eps(F_\eps(r,y)) = U(r) + \eps \Psi (r,y),
\end{equation}
and from the previous discussion we deduce that
\begin{equation}
    \label{eq:change.variable-1}
    \frac{\int_{\mathcal{C}(Y)} w_\eps \lp_{g_\eps} w_\eps \, d\mu_{g_\eps}}{ \Big( \int_{\mathcal{C}(Y)} w_{\eps}^{\frac{2n}{n-2}} \, d\mu_{g_\eps} \Big)^{\frac{n-2}{n}}}
    =
    \frac{\int_{\mathcal{C}(Y)} u_\eps \lp_g u_\eps \, d\mu_g}{ \Big( \int_{\mathcal{C}(Y)} u_{\eps}^{\frac{2n}{n-2}} \, d\mu_g \Big)^{\frac{n-2}{n}}},
\end{equation}
where we also used the simple identity $(F_\eps)_{*} \mu_{g_\eps} = \eps^{-n} \mu_g$ concerning 
push-forwards.

\subsection{Construction of the correction $\Psi$} \label{ss:c-psi}

As we will see in Section \ref{s:exp-II}, the conformal Laplacian 
expands as 
\begin{equation}
    \label{eq:expansion-01}
    \lp_{g_\eps} = \lp_0 + \eps \lp_1 +  o(\eps),
\end{equation}
and we are interested in finding a solution to the approximate Yamabe equation 
\[
\lp_{g_\eps} (U + \eps \Psi) = \mathcal{Y}_P (U + \eps \Psi)^{\frac{n+2}{n-2}} + o(\eps). 
\]
Therefore on $\Psi$ we will require that 
\begin{align}
    \label{eq:equation-Psi}
    \widetilde{\lp} \Psi & := -a \Delta_{g_0} \Psi - \mathcal{Y}_P \frac{n+2}{n-2} U^{\frac{4}{n-2}} \Psi = F, \\
    F & = - \lp_1 U \label{eq:FU}
\end{align}

From asymptotic expansions performed later, see Lemma \ref{l:errori-precisi}, we have that 
\begin{equation}
      \label{eq:L1copy}
    \lp_1 U = -\frac{a}{2} \big( \mathrm{tr}_{h_0} \xi \big) U' - \Big( \Delta_{h_0} \mathrm{tr}_{h_0} \xi - \delta^2_{h_0} \xi + 2(n-1) \mathrm{tr}_{h_0} \xi \Big) \frac{U}{r}.
\end{equation}

\begin{remark} \label{r:annull-phi1}
We notice that from \eqref{eq:L1copy} we can simply take $\Psi \equiv 0$ when both 
$\mathrm{tr}_{h_0} \xi = 0$ and $\mathsf{d} R_{h_0}[\xi] =0$, see~\eqref{eq:scalar-differential}. 
\end{remark}

    \begin{lemma}\label{l:Psi}
        Let $n \ge 3$ and $(Y,h_0)$ be a closed Einstein $(n-1)$-manifold with Einstein constant equal to $n-2$, which is not isometric to the standard sphere.      
        Let $F \in L^{\frac{2n}{n+2}} (\mathcal{C}(Y), d\mu_{g_0})$ be a function such that 
            \begin{equation}\label{eq:orth}
                \int_{\mathcal{C}(Y)} F V\, d \mu_{g_0}  = 0,
            \end{equation}
        where $V$ is the function defined by~\eqref{eq:V-function}.
        Then, there exists a unique solution $\Psi \in \mathscr{D}^{1,2}_{g_0}$ to  equation~\eqref{eq:equation-Psi} that satisfies also the condition~\eqref{eq:orth}. Moreover, by elliptic regularity the function~$\Psi$ is of class $W^{2,\frac{2n}{n+2}}_{\mathsf{loc}}$ on the regular part of $\mathcal{C}(Y)$.
    \end{lemma}

\begin{remark}
    If $F(r,y) = q(r) H(y)$ for a function $H$ with zero average on the link, then any solution of equation~\eqref{eq:equation-Psi} satisfies 
        \begin{equation}\label{eq:ortho-Psi}
            \int_{\mathcal{C}(Y)} \Psi U^{\frac{n+2}{n-2}} \, d \mu_{g_0}  = 0.
        \end{equation}
    Indeed, using the identity $\widetilde{\lp} U = - 4 \mathcal{Y}_P U^{\frac{n+2}{n-2}} / (n-2)$ and the fact that $\widetilde{\lp}$ is self-adjoint with respect to $L^2(\mathcal{C}(Y),d\mu_{g_0})$ we obtain
        \begin{align*}
            - \frac{4 \mathcal{Y}_P}{n-2}\int_{\mathcal{C}(Y)} \Psi U^{\frac{n+2}{n-2}} \, d\mu_{g_0}
            = \int_{\mathcal{C}(Y)} U \widetilde{\lp} \Psi \, d\mu_{g_0}
            = \int_0^{+\infty} U(r) q(r) r^{n-1} \, dr \int_Y H(y)  \, d\mu_{h_0} = 0.
        \end{align*}
\end{remark}

\begin{proof}
    We preliminarily notice that the function $V$ defined by~\eqref{eq:V-function} belongs to $L^{\frac{2n}{n-2}}(\mathcal{C}(Y), d\mu_{g_0})$, in particular if $F$ belongs to $L^{\frac{2n}{n+2}}(\mathcal{C}(Y), d\mu_{g_0})$, then the product between $F$ and $V$ is integrable and  the condition~\eqref{eq:orth} makes sense.
    A function $\Psi \in \mathscr{D}^{1,2}_{g_0}$ is a (weak) solution of~\eqref{eq:equation-Psi} if and only if
        \begin{equation} \label{eq:weak-sol-linearized}
            \int_{\mathcal{C}(Y)} a \langle \nabla_{g_0} \Psi, \nabla_{g_0} \Phi \rangle \, d \mu_{g_0}
            - \mathcal{Y}_P \frac{n+2}{n-2}
            \int_{\mathcal{C}(Y)} U^{\frac{4}{n-2}} \Psi \Phi \, d \mu_{g_0}
            =
            \int_{\mathcal{C}(Y)} F  \Phi \, d \mu_{g_0},
        \quad
        \forall \Phi \in \mathscr{D}^{1,2}_{g_0}.
        \end{equation}
    Now, the function
        \begin{align}
            \notag
            \mathscr{D}^{1,2}_{g_0} \times \mathscr{D}^{1,2}_{g_0} & \to \mathbb{R} \\
            \label{eq:quadratic-form-T}
            (\Psi, \Phi) & \mapsto \mathcal{Y}_P \frac{n+2}{n-2} \int_{\mathcal{C}(Y)} U^{\frac{4}{n-2}} \Psi \Phi \, d \mu_{g_0},
        \end{align}
    defines a positive, continuous and symmetric bilinear form thanks to the fact that $\mathscr{D}^{1,2}_{g_0}$ embeds in $L^2(U^{\frac{4}{n-2}} d \mu_{g_0} )$.
    Moreover, the previous embedding is compact, combining this fact with the Riesz representation theorem we deduce that there exists a positive, continuous, self-adjoint (with respect to the standard scalar product of $\mathscr{D}^{1,2}_{g_0}$), compact and linear operator $T$ from $\mathscr{D}^{1,2}_{g_0}$ to itself such that
        \[
            \langle T \Psi, \Phi \rangle_{\mathscr{D}^{1,2}_{g_0}}
            = \mathcal{Y}_P \frac{n+2}{n-2}
            \int_{\mathcal{C}(Y)} U^{\frac{4}{n-2}} \Psi \Phi \, d \mu_{g_0},
            \quad
            \forall \Psi, \Phi \in \mathscr{D}^{1,2}_{g_0}.
        \]
    Similarly, there exists $Q \in \mathscr{D}^{1,2}_{g_0}$ such that
        \[
            \langle Q, \Phi \rangle_{\mathscr{D}^{1,2}_{g_0}}
            =
            \int_{\mathcal{C}(Y)} F \Phi \, d \mu_{g_0},
            \qquad
            \forall \Phi \in \mathscr{D}^{1,2}_{g_0}.
        \]
    In particular, there exists a solution $\Psi$ to~\eqref{eq:weak-sol-linearized} if and only if there exists $\Psi \in \mathscr{D}^{1,2}_{g_0}$ such that
        \begin{equation} \label{eq:Fredholm-1}
            a \Psi - T \Psi = Q.
        \end{equation}
    
    By Fredholm's alternative (see e.g. Theorem~6.6 in~\cite{Brezis}) there exists a solution to~\eqref{eq:Fredholm-1} if and only if $Q$ is orthogonal to $\mathrm{Ker} ( a \mathrm{Id} - T)$ in $\mathscr{D}^{1,2}_{g_0}$. It is clear from the definition of $T$ that $\Phi \in \mathrm{Ker} ( a \mathrm{Id} - T)$ if and only if $\widetilde{\lp} \Phi = 0$ in $\mathscr{D}^{1,2}_{g_0}$. Therefore, thanks to Lemma~\ref{lemma:kernel-linearized} we have only to check that $Q$ is orthogonal to $V$ in $\mathscr{D}^{1,2}_{g_0}$, and this follows immediately from assumption~\eqref{eq:orth} and the definition of $Q$.
\end{proof}

We notice that if $n \ge 5$ the right-hand side of~\eqref{eq:L1copy} 
satisfies the integrability assumptions of the previous lemma. Moreover, 
it is always possible via a conformal change of metric to obtain 
condition \eqref{eq:orth}, see the comments after \eqref{eq:xi-zero-mean}. 
Therefore, with this extra property we can find a solution of $\widetilde{\lp} \Psi = -\lp_1 U$. At this point, with this choice of $\Psi$, we claim that $I(\eps)$ defined by~\eqref{eq:Yamabe-quotient-ev} extends to a regular function up to $\eps=0$ with the following expansion near zero: this will be rigorously proved in the next section. 

    \begin{proposition}\label{p:exp-full}
        Let $n \ge 5$: then we have that $I(0) = \mathcal{Y}_P$, $I'(0)=0$, and that
            \begin{equation} \label{eq:expansion-5}
                \frac{1}{2} I''(0) = - \int_{\mathcal{C}(Y)} \Psi \widetilde{\lp} \Psi \, d\mu_{g_0}
                + \frac{\omega}{2} \mathrm{Vol}_{h_0}(Y)^{\frac{n-3}{n-1}} \mathcal{R}''(h_0)[\xi,\xi]
                + \frac{\omega}{4} \int_Y \Big( (\mathrm{tr} \xi )^2 - \abs{\xi}^2 \Big) \, d \mu_{h_0},
            \end{equation}
        where $\Psi$ is as in~\eqref{eq:equation-Psi} (whose existence is ensured by Lemma~\ref{l:Psi}) and $\omega$ is the constant given by
            \begin{equation} \label{eq:cost-omega}
                \omega:=\int_0^{+\infty} U(r)^2 r^{n-1} \, dr.
            \end{equation}
\end{proposition}

\begin{remark}
    The integral in \eqref{eq:cost-omega} is not finite when $n=4$. 
    As we explained, this case will be treated separately in  Subsection \ref{ss:n=4}. 
\end{remark}

\smallskip

For later purposes we prove the following quantitative result on the solution of $\widetilde{\lp} \Psi = -\lp_1 U$ given by Lemma~\ref{l:Psi}.

\begin{lemma}\label{l:bd-Psi}
    Let $\Psi$ as above, then there exist $C, r_0 > 0$ such that for any $r \ge r_0$ we have
    \begin{equation}\label{eq:est-Psi}
    \abs{\Psi} \leq C r^{-(n-3)}, \quad \abs{\nabla_{g_0} \Psi}_{g_0} \leq 
    C r^{-(n-2)} \quad
    \text{and} \quad \abs{\nabla^2_{g_0} \Psi}_{g_0} \leq 
    C r^{-(n-1)}. 
    \end{equation}
\end{lemma}

\begin{proof}
To understand the behaviour of $\Psi$ near infinity it is convenient to use a Kelvin inversion, defining 
    \[
        \bar{\Psi}(\rho,y) = \rho^{-(n-2)} \Psi(\rho^{-1}, y).
    \]
Then it is a standard fact, due to the conformal covariance of $\lp_{g_0}$ on $\mathcal{C}(Y)$ (which is scalar-flat), that 
\begin{equation*}
\Delta_{g_0} \bar{\Psi}(\rho,y) = \rho^{-(n+2)} (\Delta_{g_0} \Psi)(\rho^{-1}, y).  
\end{equation*}
From~\eqref{eq:L1copy} and the compactness of $Y$ we easily deduce that there exist $C_0, r_0 > 0$ such that 
\[
    \abs{\Delta_{g_0} \Psi (r,y)} \le C_0 \big( r^{-(n-1)} + r^{-4} \abs{\Psi(r,y)} \big), \ \text{for any $r \ge r_0$}.
\]
Therefore $\bar{\Psi}$ satisfies 
\begin{align*}
    \abs{\Delta_{g_0} \bar{\Psi}(\rho, y)} & \le C_0
    \big( \rho^{-3} + \rho^{-(n-2)} \abs{\Psi(\rho^{-1},y)} \big) \\[1ex]
    & \le C_0
    \big( \rho^{-3} + \abs{\bar{\Psi}(\rho^{-1},y)} \big), \ \text{for any $\rho \le r_0^{-1}$}.     
\end{align*}
Since $\Psi$ is of class $\mathscr{D}^{1,2}_{g_0}$, then $\bar{\Psi}$ is in this class too. Using standard elliptic regularity theory 
it follows that there exists $C > 0$ such that 
\[
    \abs{\bar{\Psi}(\rho,y)} \leq C \rho^{-1}, \ \text{for any $\rho \le r_0^{-1}$},
\]
which corresponds to the first estimate in~\eqref{eq:est-Psi}. To get the second and the third ones, it is 
possible for example to scale variables in the equation for $\Psi$ from $B_{2R} \setminus B_R$ to $B_2 \setminus B_1$ for 
$R$ large, use elliptic regularity in this 
fixed annulus, and  scale the variables back. 
\end{proof}

\begin{remark} \label{rmk:continuita-Psi}
    Near the origin the right-hand side of~\eqref{eq:L1copy} is of class $L^p$ for any $p < n$. This implies that $\Psi$ is continuous on $\mathcal{C}(Y)$, see e.g. Proposition 2.7 in \cite{Behrndt}. Combining this fact with the first estimate in~\eqref{eq:est-Psi} we deduce that there exists $C > 0$ such that
    \begin{equation} \label{eq:cont-Psi-est}
        \abs{\Psi(r,y)} \le C (1+r) U(r), \ \text{for every $r > 0$}.
    \end{equation}
\end{remark}

\subsection{Role of the correction $\Psi$ in the expansion}\label{ss:role}

The goal of this section is to prove that the  expression  in the right-hand side of \eqref{eq:expansion-5} is negative-definite 
under the assumptions on $(Y,h_0)$ and on $\xi$ in Theorem~\ref{t:main}.


\smallskip 

Using the decomposition~\eqref{eq:decomposition-xi} and Proposition~\ref{prop:div-hess} we can rewrite the right-hand side of equation~\eqref{eq:L1copy} in the following way:
    \begin{equation} \label{eq:L1-2}
        \lp_1 U 
        =
        -\frac{2(n-1)}{n-2} U' \big( \Delta G + (n-1) f \big) - \frac{U}{r} \Big( (n-2) \Delta (f - G) + 2(n-1) \big( \Delta G + (n-1) f \big) \Big). 
    \end{equation}

As anticipated in the introduction, we want to split the solution of the equation $\widetilde{\lp} \Psi = - \lp_1 U$ in the sum of two terms, one of which explicit.

In order to find  a proper splitting, let us first consider the case $G = f$. In this situation the right-hand side of~\eqref{eq:L1-2} and the equation for $\Psi$ simplify to
    \begin{equation} \label{eq:eq-Psi-simplified}
        \widetilde{\lp} \Psi = 2(n-1) \bigg( \frac{U'}{n-2} + \frac{U}{r} \bigg) \Big( \Delta f + (n-1) f \Big),
    \end{equation}
which admits, by a direct computation, the explicit solution
    \begin{equation} \label{eq:explicit-Psi1}
    \Psi_1(r,y) := -\frac{n-2}{2} \bigg( \frac{U'(r)}{n-2} + \frac{U(r)}{r} \bigg) r^2 f(y), \quad \Psi_1 \in \mathscr{D}^{1,2}_{g_0}.
    \end{equation}

Now, we turn back to the general case and we write $\Psi = \Psi_1 + \Psi_2$, where $\Psi_1$ is the function defined by~\eqref{eq:explicit-Psi1} and $\Psi_2$ is unknown. Combining the identity~\eqref{eq:L1-2} with the fact that $\Psi_1$ is a solution of equation~\eqref{eq:eq-Psi-simplified} we obtained that $\Psi_2$ must solve the equation 
    \begin{equation} \label{eq:eq-Psi2}
         \widetilde{\lp} \Psi_2 = \bigg( \frac{2(n-1)}{n-2} U' + n \frac{U}{r} \bigg) \Delta (G - f).
    \end{equation}

For convenience we introduce a notation for the radial part of the functions appearing on the right-hand sides of equations~\eqref{eq:eq-Psi-simplified} and~\eqref{eq:eq-Psi2}:
\begin{equation}
    \label{eq:parti-radiali}
    q_1(r):=\frac{U'}{n-2} + \frac{U}{r} \quad \text{and} \quad q_2(r) := \frac{2(n-1)}{n-2} U' + n \frac{U}{r}.
\end{equation}

We want to estimate the first addend in the right-hand side of the identity~\eqref{eq:expansion-5}. Starting from~\eqref{eq:eq-Psi-simplified} and~\eqref{eq:explicit-Psi1}, it is easy to compute
    \begin{equation} \label{eq:contributo-Psi1}
        \int_{\mathcal{C}(Y)} \Psi_1 \widetilde{\lp} \Psi_1 \, d\mu_{g_0} = (n-1)(n-2) \int_0^{+\infty} q_1(r)^2 r^{n+1} \, dr \int_Y \Big( \abs{\nabla f}^2 - (n-1) f^2 \Big) \, d \mu_{h_0}.
    \end{equation}
In particular, using the identities stated in Proposition~\ref{prop:int-iden} we have
    \begin{equation} \label{eq:coefficiente-q1}
        \int_0^{+\infty} q_1(r)^2 r^{n+1}
        =
        \int_0^{+\infty} \bigg( \frac{\big(U'\big)^2 r^2}{(n-2)^2} + \frac{2 UU' r}{n-2} + U^2 \bigg) r^{n-1} \, dr = \frac{n-4}{4(n-1)} \omega,
    \end{equation}
where $\omega$ is the constant defined by~\eqref{eq:cost-omega}.

Moreover, $\Psi_1$ and $\Psi_2$ are orthogonal with respect to $\widetilde{\lp}$, indeed we have
\begin{equation}
    \label{eq:Psi1-Psi2}
    \int_{\mathcal{C}(Y)} \Psi_1 \widetilde{\lp} \Psi_2 \, d\mu_{g_0} = \frac{n-2}{2} \int_0^{+\infty} q_1(r) q_2(r) r^{n+1} \, dr \int_Y \Big( \abs{\nabla f}^2 - \nabla f \cdot \nabla G \Big) \, d \mu_{h_0},
\end{equation}
and turning our attention to the radial factor we notice that
\begin{align}
    \notag
    \int_0^{+\infty} q_1(r) q_2(r) r^{n+1} \, dr & = \int_0^{+\infty} \bigg( n U^2 + \frac{2(n-1)}{(n-2)^2} \big( U' \big)^2 r^2 + \Big( 1 + \frac{2n}{n-2} \Big) U U' r \bigg) r^{n-1} \, dr \\[1.5ex]
    \label{eq:radial}
    & = \underbrace{\bigg( n + \frac{2(n-1)}{(n-2)^2} \frac{n(n^2-4)}{4(n-1)} - \Big( 1 + \frac{2n}{n-2} \Big) \frac{n}{2}  \bigg)}_{=0} \omega = 0,
\end{align}
where we used again the integral identities in Proposition~\ref{prop:int-iden}.

Therefore, from~\eqref{eq:Psi1-Psi2} and~\eqref{eq:radial} we obtain
    \begin{equation} \label{eq:decomp-ortho-Psi}
        \int_{\mathcal{C}(Y)} \Psi \widetilde{\lp} \Psi \, d\mu_{g_0}
        =
        \int_{\mathcal{C}(Y)} \Psi_1 \widetilde{\lp} \Psi_1 \, d\mu_{g_0} + \int_{\mathcal{C}(Y)} \Psi_2 \widetilde{\lp} \Psi_2 \, d\mu_{g_0}.
    \end{equation}
Now, we want to estimate the contribution of the second addend in the right-hand side of the previous identity. Contrary to what happened for $\Psi_1$ there is not a simple and explicit formula for $\Psi_2$. However, combining Remark~\ref{rmk:fourier-D} with the fact that $\Psi_2$ is a solution of equation~\eqref{eq:eq-Psi2} it is easy to deduce that we can write
    \begin{equation} \label{eq:def-Psi2}
        \Psi_2(r,y) := \sum_{k=0}^{+\infty} c_k \psi_{\lambda_k}(r) u_k(y),
    \end{equation}
where $\{ u_k \}_{k \in \mathbb{N}}$ is an orthonormal basis of $L^2(Y, d\mu_{h_0})$ such that $-\Delta_{h_0} u_k = \lambda_k u_k$, for every $k \ge 1$ the function $\psi_{\lambda_k}$ is a solution of the following second order ordinary differential equation
    \begin{equation} \label{eq:psi-k}
        \mathcal{L}_{\lambda} \psi = \lambda a^{-1} q_2(r), \quad \lambda=\lambda_k,
    \end{equation}
where
    \begin{equation} \label{eq:psi-k-2}
        \mathcal{L}_{\lambda} \psi := - \psi''(r) - \frac{n-1}{r} \psi'(r) + \frac{\lambda}{r^2} \psi(r) - \frac{n(n+2)}{(1+r^2)^2} \psi(r),
    \end{equation}
and the coefficients $c_k$ are such that 
    \begin{equation} \label{eq:decomp-f-G}
        f - G = \sum_{k = 0}^{+\infty} c_k u_k.
    \end{equation}
We notice that $c_0 = 0$ because $f$ and $G$ have zero average, see \eqref{eq:decomposition-xi}, \eqref{eq:xi-zero-mean} and the lines after these formulas. 

Moreover, it is clear from the definition of $\Psi_2$ given by equation~\eqref{eq:eq-Psi2} and the decomposition~\eqref{eq:def-Psi2} that
    \begin{equation} \label{eq:key-estimate-Psi-2}
        \int_{\mathcal{C}(Y)} \Psi_2 \widetilde{\lp} \Psi_2 \, d\mu_{g_0}
        =
        \int_{\mathcal{C}(Y)} \Psi_2 q_2 \Delta (G-f) \, d\mu_{g_0} = \sum_{k = 1}^{+\infty} \lambda_k c_k^2 \underbrace{ \int_0^{+\infty} \psi_{\lambda_k}(r) q_2(r) r^{n-1} \, d r}_{=: \beta_k}.
    \end{equation}

At this point the following proposition is useful.
    \begin{proposition} \label{prop:monotonicity-beta}
    Let $n \ge 5$ and let $(Y,h_0)$ be a closed Einstein $(n-1)$-manifold with Einstein constant equal to $n-2$.
    For every integer number $k \ge 1$ we define the quantity
        \begin{equation} \label{eq:def-beta-k}
            \beta_k := \int_0^{+\infty} \psi_{\lambda_k}(r) q_2(r) r^{n-1} \, d r,
        \end{equation}
    where $\{ \lambda_k \}$ is the sequence of eigenvalues of $-\Delta_{h_0}$ and $\psi_{\lambda_k}$ and $q_2$ are defined respectively by~\eqref{eq:psi-k-2} and by~\eqref{eq:parti-radiali}.
    Then, $\{ \beta_k \}_{k \ge 1}$ is a strictly increasing sequence and
        \begin{equation} \label{eq:beta-1-ineq}
            \beta_1 \ge \frac{n-2}{4} \omega,
        \end{equation}
    where $\omega$ is the constant defined by~\eqref{eq:cost-omega}.
    
    Moreover, the equality case in~\eqref{eq:beta-1-ineq} holds if and only if $(Y,h_0)$ is isometric to the $(n-1)$-dimensional sphere endowed with its standard metric. 
    \end{proposition}

In order to prove the previous proposition we need the following two lemmas.
    \begin{lemma} \label{lemma:exist-ortho}
        Let $n \ge 5$: then for every $\lambda \ge n-1$ there exists a unique solution $\psi_{\lambda} \in \mathscr{D}^{1,2}_{g_0}$ of equation~\eqref{eq:psi-k} such that
            \begin{equation} \label{eq:ortho-weighted}
                \int_0^{+\infty} \psi_{\lambda}(r) U'(r) r^{n-3} \, dr = 0.
            \end{equation}
    \end{lemma}

    \begin{proof}
        If $\lambda=\lambda_k$ and $(Y,h_0)$ is different from the standard sphere, then the existence and uniqueness of a solution $\psi_{\lambda} \in \mathscr{D}^{1,2}_{g_0}$ of equation~\eqref{eq:psi-k} is an immediate consequence of Remark~\ref{rmk:fourier-D} and Lemma~\ref{l:Psi} applied with $F(r,y) = q_2(r) u_k(y)$. We notice that by Lichnerowicz's eigenvalue estimate we have $\lambda_k \ge \lambda_1 > n-1$.

        For a generic $\lambda \ge n-1$ the argument to prove the existence of a solution is very similar to that used in Lemma~\ref{l:Psi}, so we will only sketch it.
        Let $\mathscr{D}^{1,2}_{g_0, \mathsf{rad}}$ be the subspace of $\mathscr{D}^{1,2}_{g_0}$ consisting of functions with radial symmetry.
        We consider the scalar product $\langle \cdot, \cdot \rangle_{\mathscr{D}_\lambda}$ on~$\mathscr{D}^{1,2}_{g_0, \mathsf{rad}}$ defined as follows:
            \begin{equation} \label{eq:equival-lambda}
                \langle \psi , \phi \rangle_{\mathscr{D}_\lambda} := \int_{0}^{+\infty} \psi'(r) \phi'(r) r^{n-1} \, dr + \lambda \int_{0}^{+\infty} \psi(r) \phi(r) r^{n-3} \, dr,
                \qquad
                \forall \psi, \phi \in \mathscr{D}^{1,2}_{g_0, \mathsf{rad}}
            \end{equation}
        which is equivalent to the standard scalar product of $\mathscr{D}^{1,2}_{g_0, \mathsf{rad}}$ thanks to Hardy's inequality.
        
         In the same way as in Lemma~\ref{l:Psi}, for every $\lambda \ge n-1$, there exists a positive continuous self-adjoint compact linear operator $T_{\lambda}$ from $\mathscr{D}^{1,2}_{g_0, \mathsf{rad}}$ to itself such that
            \[
                \langle T_{\lambda} \psi, \phi \rangle_{\mathscr{D}_\lambda}
                =
                n(n+2) \int_0^{+\infty} \frac {\psi(r) \phi(r)}{(1+r^2)^2} r^{n-1} \, dr,
                \qquad
                \forall \psi, \phi \in \mathscr{D}^{1,2}_{g_0, \mathsf{rad}}.
            \]

        It is clear by the definition~\eqref{eq:equival-lambda} that $\norm{T_{\lambda}}_{\mathscr{D}_\lambda} < \norm{T_{n-1}}_{\mathscr{D}_{n-1}}$ for every $\lambda > n-1$, where $\norm{ \, \cdot \,}_{\mathscr{D}_\lambda}$ denotes the operator norm induced by $\langle \cdot, \cdot \rangle_{\mathscr{D}_\lambda}$. Moreover, we claim that $\norm{T_{n-1}}_{\mathscr{D}_{n-1}} = 1$, indeed 
        from Remark~\ref{rmk:U'-w-eing} it is easy to check that $T_{n-1} U' = U'$, and combining this fact with the property of $U'$ having constant sign we deduce that
            \[
                \norm{T_{n-1}}_{\mathscr{D}_{n-1}} = \max_{\substack{\psi \in \mathscr{D}^{1,2}_{g_0, \mathsf{rad}} \\ \psi \neq 0}} \frac{\langle T_{n-1} \psi, \psi \rangle_{\mathscr{D}_{n-1}}}{\norm{\psi}_{\mathscr{D}_{n-1}}^2} = \frac{\langle T_{n-1} U', U' \rangle_{\mathscr{D}_{n-1}}}{\norm{U'}_{\mathscr{D}_{n-1}}^2} = 1.
            \]
        We remark here that the fact that $U'$ has constant sign in $(0,+\infty)$ implies that $\mu = 1$ is a simple (and also the largest) eigenvalue of $T_{n-1}$, in particular the kernel of the operator $\mathrm{Id} - T_{n-1}$ is spanned by $U'$. 

        The previous discussion implies that for every $\lambda > n-1$ the operator $\mathrm{Id}-T_{\lambda}$ is self-adjoint and coercive in $\mathscr{D}^{1,2}_{g_0, \mathsf{rad}}$ and thus invertible. This is equivalent to say that for every $\lambda > n-1$ there exists a unique solution $\psi_{\lambda} \in \mathscr{D}^{1,2}_{g_0, \mathsf{rad}}$ of equation~\eqref{eq:psi-k}. 

        Now, we prove that property~\eqref{eq:ortho-weighted} holds for $\lambda > n-1$. From Remark~\ref{rmk:U'-w-eing} it is clear that
            \[
                \mathcal{L}_{\lambda} U' = (\lambda - n+1 ) \frac{U'}{r^2}
            \]
        and combining this with the fact that $\mathcal{L}_{\lambda}$ is self-adjoint in $L^2(r^{n-1} dr)$ we obtain
            \[
                \int_0^{+\infty} \psi_{\lambda} U' r^{n-3} \, dr = \frac{1}{\lambda - n+1 } \int_0^{+\infty} \mathcal{L}_{\lambda} \psi_{\lambda} U' r^{n-1} \, dr
                =
                \frac{\lambda}{a(\lambda - n+1)} \int_0^{+\infty} q_2 U' r^{n-1} \, dr = 0,
            \]
        where we used that $\psi_{\lambda}$ is the solution of~\eqref{eq:psi-k} and a direct computation for the last identity.

        We are left with the case $\lambda = n-1$. Let $p \in \mathscr{D}^{1,2}_{g_0, \mathsf{rad}}$ be such that
            \[
                \langle p, \phi \rangle_{\mathscr{D}_{n-1}}
                =
                \frac{n-2}{4} \int_0^{+\infty} q_2 \phi r^{n-1} \, dr,
                \qquad
                \forall \phi \in \mathscr{D}^{1,2}_{g_0, \mathsf{rad}},
            \]
        then by Fredholm's alternative there exists a solution of $\psi - T_{n-1} \psi = p$ if and only if $p$ is orthogonal to the kernel of $\mathrm{Id}-T_{n-1}$ in $\mathscr{D}^{1,2}_{g_0, \mathsf{rad}}$. Therefore, for what we have already seen we have only to check that $p$ is orthogonal to $U'$ in $\mathscr{D}^{1,2}_{g_0, \mathsf{rad}}$, and this follows immediately from the definition of $p$ and a direct computation. Clearly, in this case the solution is not unique, but it is unique if we want property~\eqref{eq:ortho-weighted} to hold.
    \end{proof}

\begin{remark} \label{remark:psi-n-1} 
    Let $\widehat{\psi}$ be the function defined by
        \begin{equation} \label{eq:psi-n-1}
            \widehat{\psi}(r)
            :=
            \bar{U}'(r) \frac{r^2 + 2 \ln r}{8},
        \end{equation}
    where $\bar{U}$ is the function introduced in~\eqref{eq:U-def}.
    Then, by a direct computation $\widehat{\psi}$ is a solution of equation~\eqref{eq:psi-k} with $\lambda = n-1$. In particular, this implies that
    $\psi_{n-1} = \widehat{\psi} + c\bar{U}'$ for a certain real number $c$.
\end{remark}

The condition~\eqref{eq:ortho-weighted} suggests to consider the following space of functions
    \begin{equation} \label{eq:def-E}
        \mathscr{E}_{\mathsf{rad}} = \bigg\{ \psi \in \mathscr{D}^{1,2}_{g_0, \mathsf{rad}} : \int_0^{+\infty} \psi(r) U'(r) r^{n-3} \, dr = 0 \bigg\}.
    \end{equation}
A remarkable property of the space $\mathscr{E}_{\mathsf{rad}}$ is that it enjoys a weighted Poincaré inequality, with respect to the weight $(1+r^2)^{-2}$, with a slightly better constant compared to that of $\mathscr{D}^{1,2}_{g_0}$. More precisely, we have the following lemma.
    \begin{lemma} \label{lemma:better-constant}
        For every $\psi \in \mathscr{E}_{\mathsf{rad}}$ it holds
            \begin{equation} \label{eq:better-constant}
                n(n+2) \int_0^{+\infty} \frac{\psi(r)^2}{(1+r^2)^2} r^{n-1} \, dr \le \int_0^{+\infty} (\psi'(r))^2 r^{n-1} \, dr.
            \end{equation}
        Moreover, the equality case in~\eqref{eq:better-constant} holds if and only if $\psi$ is a multiple of the function $V$ defined by~\eqref{eq:V-function}.
    \end{lemma}

    \begin{remark}
    In general for functions $\psi \in \mathscr{D}^{1,2}_{g_0}$ an inequality of the type of~\eqref{eq:better-constant} still holds but with the worse constant $n(n-2)$, see e.g. Proposition~2 in~\cite{BBDGV} or Lemma~A1 in~\cite{Bianchi-Egnell}.
    \end{remark}

    \begin{proof}
        Let $\mathcal{L}_0$ be the operator defined by~\eqref{eq:psi-k-2} with $\lambda = 0$, which coincides with the radial part of the linearized operator $ a^{-1} \widetilde{\lp}$ defined by~\eqref{eq:V-ker-L}. We know from Remark~\ref{rmk:U'-w-eing} that $U'$ is an eigenfunction of $r^2 \mathcal{L}_0$ with corresponding eigenvalue equal to $-(n-1)$. Moreover, the fact that $U'$ does not change sign in $(0,+\infty)$ implies that it is the first eigenfunction in $L^2(r^{n-3} dr)$. In particular, from the definition of $\mathscr{E}_{\mathsf{rad}}$ it follows that for every $\psi \in \mathscr{E}_{\mathsf{rad}}$ we have
            \begin{equation*} 
                \int_0^{+\infty} \bigg( (\psi'(r))^2 - 
                n(n+2) \frac{\psi(r)^2}{(1+r^2)^2} \bigg)
                r^{n-1} \, dr \ge \mu_2 \int_0^{+\infty} \psi(r)^2 r^{n-3} \, dr,
            \end{equation*}
        where $\mu_2$ is the second eigenvalue of $r^2 \mathcal{L}_0$.

        At this point we claim that $\mu_2 = 0$, from which the conclusion will follow. We observe that $r^2 \mathcal{L}_0 V = 0$, where $V$ is the function defined by~\eqref{eq:V-function}, and also that the set $\{ V \neq 0 \}$ has exactly two connected components. From the Sturm-Liouville theory (see for example~\cite{Courant-Hilbert}, pages 454-455, whose 
        argument can be easily adapted to our case) this implies that $V$ is an eigenfunction corresponding to the second eigenvalue $\mu_2$, and therefore $\mu_2 = 0$. 

        Moreover, if $\psi \in \mathscr{E}_{\mathsf{rad}}$ is a function that satisfies the equality case in~\eqref{eq:better-constant}, then necessarily $r^2 \mathcal{L}_0 \psi = 0$ which is equivalent to $\widetilde{\lp} \psi = 0$. Finally, from Lemma~\ref{lemma:kernel-linearized} and the discussion above it we know that $\psi$ must be a multiple of $V$.
    \end{proof}


\begin{proof}[Proof of Proposition~\ref{prop:monotonicity-beta}]
    For every $\lambda \ge n-1$ we define the quantity
    \begin{equation} \label{eq:def-beta-lambda}
            \beta(\lambda) := \int_0^{+\infty} \psi_{\lambda}(r) q_2(r) r^{n-1} \, d r,
    \end{equation}
    where $\psi_{\lambda}$ is the function given by Lemma~\ref{lemma:exist-ortho}.
    In this way we have that $\beta_k = \beta(\lambda_k)$. Now, we prove that the function $\beta$ is strictly increasing.
    
    It is well known that from the existence and uniqueness of $\psi_{\lambda}$ it follows that 
    the map $\lambda \mapsto \psi_{\lambda}$ is differentiable with respect to $\lambda$; we will denote its derivative with the symbol $v_{\lambda}$. Now, differentiating both sides of equation~\eqref{eq:psi-k} with respect to $\lambda$ we obtain that $v_{\lambda}$ is a solution of the following equation:
        \begin{equation} \label{eq:v-lambda}
            \mathcal{L}_{\lambda} v_{\lambda} = a^{-1} q_2 - \frac{\psi_{\lambda}}{r^2} =
            \frac{\mathcal{L}_{\lambda} \psi_{\lambda}}{\lambda}
            - \frac{\psi_{\lambda}}{r^2}.
        \end{equation}

    Therefore, the function $\beta$ is also differentiable with respect to $\lambda$ and we can compute its derivative, finding 
        \begin{equation} \label{eq:mono-1}
            \beta'(\lambda) = \int_0^{+\infty} v_{\lambda} q_2 r^{n-1} \, d r 
            =
            \frac{a}{\lambda} \int_0^{+\infty} v_{\lambda} \mathcal{L}_{\lambda} \psi_{\lambda} r^{n-1} \, d r
            =
            \frac{a}{\lambda} \int_0^{+\infty} \psi_{\lambda} \mathcal{L}_{\lambda} v_{\lambda} r^{n-1} \, d r,
        \end{equation}
    where to deduce the last identity we used that $\mathcal{L}_{\lambda}$ is self-adjoint in $L^2(r^{n-1} dr)$.
        
    From~\eqref{eq:v-lambda} we derive
        \begin{align} \notag
            \int_0^{+\infty} \psi_{\lambda} \mathcal{L}_{\lambda} v_{\lambda} r^{n-1} \, d r
            & =
            \frac{1}{\lambda} \int_0^{+\infty} \psi_{\lambda} \mathcal{L}_{\lambda} \psi_{\lambda} r^{n-1} \, dr -
            \int_0^{+\infty} \psi_{\lambda}^2 r^{n-3} \, dr \\[1ex]
            \label{eq:mono-2}
            & =
            \frac{1}{\lambda} \int_0^{+\infty} \! \bigg( ( \psi_{\lambda}' )^2 - \frac{n(n+2)}{(1+r^2)^2} \psi_{\lambda}^2 \bigg) r^{n-1} \, dr.
        \end{align}
    Combining~\eqref{eq:mono-1},~\eqref{eq:mono-2} and Lemma~\ref{lemma:better-constant} we deduce that $\beta'(\lambda) \ge 0$ with equality if and only if $\psi_{\lambda}$ is a multiple of $V$, but this is never the case since $\mathcal{L}_{\lambda} V$ is not a multiple of $q_2$. Therefore, $\beta'(\lambda) > 0$ and~$\beta$ is strictly increasing.

    Finally, using the explicit expression of $\psi_{n-1}$ given in Remark~\ref{remark:psi-n-1} it is easy to check that
        \[
            \beta(n-1) = \frac{n-2}{4} \omega.
        \]
    Moreover, if $\beta_1 = \beta(n-1)$, then by the strict monotonicity of $\beta$ we deduce that $\lambda_1 = n-1$ and by the Lichnerowicz-Obata theorem  this implies that $(Y,h_0)$ is isometric to the $(n-1)$-dimensional sphere endowed with its standard metric.
\end{proof}

\smallskip

We now comment on the condition 
\begin{equation} \label{eq:xi-zero-mean}
    \int_Y \mathrm{tr} \xi \, d\mu_{h_0} = 0,
\end{equation}
which has been used so far in Proposition \ref{prop:second-var-EH} and in 
\eqref{eq:R-parte-conforme}. Via a conformal change of the form 
$g \mapsto \Tilde{g} = \mathfrak{u} g$ with 
$\mathfrak{u} (s, y) = 1 + 2 s f(y)$ it is possible to show, 
re-defining the geodesic distance for $\tilde{g}$, that $\xi$ 
transforms as 
\begin{equation} \label{eq:transf-xi}
\xi \longmapsto \xi + \nabla^2 f + f h_0.
\end{equation}
In this way, taking as $f$ a proper constant on $Y$, we 
can always reduce ourselves to condition \eqref{eq:xi-zero-mean}.

Finally, using Proposition~\ref{prop:monotonicity-beta} we are able to conclude the proof of our claim:

    \begin{proposition} \label{p:neg-def}
        If $n \ge 5$ and if we are under the assumptions of Theorem~\ref{t:main}, then the quadratic form defined by the right-hand side of~\eqref{eq:expansion-5} is negative-definite. 
    \end{proposition}

    \begin{proof}
Taking into account the decomposition given in~\eqref{eq:decomposition-xi}, we can rewrite the third addend in~\eqref{eq:expansion-5} as
\begin{align}
    \notag
    \int_Y \Big( (\mathrm{tr} \xi )^2 - \abs{\xi}^2 \Big) \, d \mu_{h_0}  = & \int_Y \bigg( (\Delta G)^2 + (n-1)(n-2) f^2 - 2(n-2) \nabla G \cdot \nabla f \bigg) \, d \mu_{h_0} \\[1ex]
    \notag
    & - \int_Y \bigg( \abs{\nabla^2 G}^2 + \abs{\delta^* \alpha}^2 + \abs{\eta}^2 \bigg) \, d \mu_{h_0} \\[1ex]
    \notag
    = & (n-2) \int_Y \bigg( \abs{\nabla (f - G)}^2 - \abs{\nabla f}^2 + (n-1) f^2 \bigg) \, d \mu_{h_0} \\[1ex]
    \label{eq:third-summand-dec}
     & -  \int_Y \bigg( \abs{ \delta^* \alpha }^2 + \abs{\eta}^2 \bigg) \, d \mu_{h_0},
\end{align}
where to derive the last equality we used  Bochner's identity and the elementary formula $\abs{\nabla (G - f)}^2 = \abs{\nabla G}^2 + \abs{\nabla f}^2 - 2 \nabla G \cdot \nabla f$.

        Next, combining~\eqref{eq:contributo-Psi1},~\eqref{eq:coefficiente-q1}, 
        the fact that $\mathcal{R}''$ is orthogonal with respect to the decomposition 
        in Lemma~\ref{lemma:decomposition-tensor},~\eqref{eq:R-parte-conforme} and~\eqref{eq:third-summand-dec} we deduce
        \begin{multline} \label{eq:simplify-quadratic-1}
                - \int_{\mathcal{C}(Y)} \Psi_1 \widetilde{\lp} \Psi_1 \, d\mu_{g_0}
                + \frac{\omega}{2} \mathrm{Vol}_{h_0}(Y)^{\frac{n-3}{n-1}} \mathcal{R}''(h_0)[\xi,\xi]
                + \frac{\omega}{4} \int_Y \Big( (\mathrm{tr} \xi )^2 - \abs{\xi}^2 \Big) \, d \mu_{h_0} \\[1ex]
                \le
                \frac{n-2}{4} \omega \int_Y \abs{\nabla (f-G)}^2 \, d \mu_{h_0},
        \end{multline}
with equality if and only if we have both $\delta^* \alpha = 0$ and $\eta = 0$.

        Therefore, if we prove that
        \begin{equation*}
            \label{eq:key-estimate-Psi}
            \int_{\mathcal{C}(Y)} \Psi_2 \widetilde{\lp} \Psi_2 \, d\mu_{g_0} > \frac{n-2}{4} \omega \int_Y \abs{\nabla (f - G)}^2 \, d \mu_{h_0},
        \end{equation*} 
        then the conclusion will follow thanks to identity~\eqref{eq:decomp-ortho-Psi}.

        For the last inequality we notice that, from~\eqref{eq:key-estimate-Psi-2} and Proposition~\ref{prop:monotonicity-beta} one has 
        \begin{equation*}
            \int_{\mathcal{C}(Y)} \Psi_2 \widetilde{\lp} \Psi_2 \, d\mu_{g_0} \ge \beta_1 \sum_{k = 1}^{+\infty} \lambda_k c_k^2 = \beta_1 \int_Y \abs{\nabla (f - G)}^2 \, d \mu_{h_0} > \frac{n-2}{4} \omega \int_Y \abs{\nabla (f - G)}^2 \, d \mu_{h_0},
        \end{equation*}
        where we used~\eqref{eq:beta-1-ineq} and the fact that $(Y,h_0)$ is not isometric to the standard sphere.
        The proof is therefore concluded. 
    \end{proof}

%% file: Section4A.tex
\section{Expansion of the Yamabe quotient} \label{s:exp-II}

In this section, using the estimates of the previous one, we will expand 
the Yamabe quotient on suitable test functions, in order to prove inequality 
\eqref{eq:sing-quot} starting with $n \geq 5$. The case $n = 4$ will be treated 
in a separate subsection, for technical reasons.


We start by considering the metric $g_\eps$ on the cone $\mathcal{C}(Y)$, as in the beginning of Section~\ref{s:exp-I}. This metric is a perturbation in $\eps$ of the pure conical metric $g_0$ near the conical point. We will write, for the conformal Laplacian and the volume form with respect to this metric:
\begin{equation}
    \label{eq:expansion-1}
    \lp_{g_\eps} = \lp_0 + \eps \lp_1 + \eps^2 \lp_2 + o(\eps^2), \quad \mu_{g_\eps} = \big(1+\eps m_1 + \eps^2 m_2+o(\eps^2)\big) \mu_{g_0}.
\end{equation}
The operators $\lp_i$ have an explicit form, which we will 
 only need for radial functions, For general functions, we will limit ourselves to display some 
 characterizations of the operators in terms of their coefficients. 
We collect some useful properties in the next lemma, together with rigorous estimates 
on the error terms. These estimates will be crucial to expand the numerator and the denominator in the Yamabe quotient of $(M,\bar{g})$ on suitable test functions in powers of $\eps$.

\begin{lemma}\label{l:errori-precisi}
Let $h(s)$ be a family of metrics as in Proposition~\ref{prop:espansione-operatori-2} and let $h_0=h(0)$,~$\xi=h'(0)$ and $\eta=h''(0)$. Let $g_\eps$ be the metric on $\mathcal{C}(Y)$ defined as $dr^2 + r^2 h(\eps r)$. Then, for any $\delta > 0$ there exists a constant $C=C(n,h,\delta)$ such that in the set $\{ r \leq 2 \delta / \eps \}$ the following expansions hold 
    \begin{equation}\label{eq:exp-mu}
     \mu_{g_\eps} = \Big( 1+ \eps m_1 + \eps^2 m_2 + \eps^3 m_3 \Big) r^{n-1} \mu_{h_0}, 
    \end{equation}
    where $m_1, m_2$ are explicit and given by 
\begin{equation} \label{eq:m1}
    m_1(r,y) = \frac{r}{2} \mathrm{tr}_{h_0} \xi,
\end{equation}
\begin{equation} \label{eq:m2}
    m_2(r,y) = \frac{r^2}{8} \Big( (\mathrm{tr}_{h_0} \xi)^2 - 2 \abs{\xi}^2_{h_0} + 2 \mathrm{tr}_{h_0} \eta  \Big),
\end{equation}
and where $m_3$ satisfies
    \begin{equation} \label{eq:m3}
        \abs{m_3(r,y)} \le C r^3.
    \end{equation}

Moreover, for every $u \in C^2 \big(\mathcal{C}(Y) \setminus \{ P \} \big)$ we have 
    \begin{equation}\label{eq:exp-Lapl}
     \lp_{g_\eps} u = \lp_0 u + \eps \lp_1 u + \eps^2 \lp_2 u 
    + \eps^3 \lp_3 u. 
    \end{equation}
    Here, the operators $\lp_i$ satisfy the following properties 
\begin{gather}
    \label{eq:L0}
    \lp_0 = -a \Delta_{g_0} = -a \Big( \partial_{rr} + \frac{(n-1)}{r} \partial_r + \frac{\Delta_{h_0}}{r^2} \Big), \\[1.5ex]
    \label{eq:L1-auto}
    \lp_1 + m_1 \lp_0 \ \text{is formally self-adjoint in $L^2(\mathcal{C}(Y),d\mu_{g_0})$}, \\[1.5ex]
   \label{eq:L3}
    \abs{\lp_i u} \le C r^{i-2} \Big( \abs{u} + r \abs{\nabla_{g_0} u}_{g_0}
     +  r^2 \abs{\nabla_{g_0}^2 u}_{g_0} \Big), \quad \text{for $i=1,2,3$}.
\end{gather}
Furthermore, we also have the following identities
\begin{gather}
    \label{eq:L1}
    \lp_1 U = -\frac{a}{2} \big( \mathrm{tr} _{h_0} \xi \big) U' - \Big( \Delta_{h_0} \mathrm{tr}_{h_0} \xi - \delta^2_{h_0} \xi + 2(n-1) \mathrm{tr}_{h_0} \xi \Big) \frac{U}{r}, \\[1.5ex]
    \label{eq:L2U}
    \lp_2 U = -a \Delta_2 U + \widehat{R}_2(\xi, \eta) U,
\end{gather}
where 
\begin{gather*}
    \Delta_2 U = \frac{1}{2} \Big( \mathrm{tr}_{h_0} \eta - \abs{\xi}_{h_0}^2 \Big) r U', \\[1.5ex]
    \widehat{R}_2(\xi, \eta) = \frac{R_2(\xi, \eta)}{2} + \bigg( \! n + \frac{3}{4} \bigg) \abs{\xi}_{h_0}^2 - \frac{1}{4} \big( \mathrm{tr}_{h_0} \xi \big)^2 - (n+1) \mathrm{tr}_{h_0} \eta,
\end{gather*}
and where $R_2(\xi, \eta)$ is defined as the right-hand side of the identity~\eqref{eq:der-scalar-2}.
\end{lemma}

\begin{proof}
    We start by proving  properties~\eqref{eq:exp-mu}-\eqref{eq:m3}. The compactness of $Y$ ensures that for every $s \ge 0$ there exists a positive function $\kappa_s$ on $Y$ such that $\mu_{h(s)} = \kappa_s \mu_{h_0}$. For any $s \le 2 \delta$, by the regularity of $h(s)$ with respect to $s$, we have
        \begin{equation} \label{eq:resto-m3}
            \abs*{ \kappa_s - 1 - c_1 s - \frac{c_2}{2} s^2 } \le \frac{s^3}{6} \max_{s \in [0,2\delta]} \, \abs*{ \frac{d^3 \kappa_s}{ds^3}(s)},
        \end{equation}
    where $c_1$ and $c_2$ are easily obtain from~\eqref{eq:measure-2} and~\eqref{eq:measure-3} and are given by
        \begin{equation*} 
            c_1 = \frac{1}{2} \mathrm{tr}_{h_0} \xi , \qquad
            c_2 = \frac{1}{4} \Big( ( \mathrm{tr}_{h_0} \xi )^2 
            - 2\abs{\xi}_{h_0}^2 + 2\mathrm{tr}_{h_0} \eta \Big).
        \end{equation*}
    The conclusion follows from~\eqref{eq:resto-m3} after combining the change of variable $s = \eps r$ with the identity~\eqref{eq:measure-1}.

    In a similar way as before, for any $u \in C^2(Y)$ and for any $s \le 2 \delta$ we have 
        \begin{equation} \label{eq:resto-Delta}
            \abs*{ \Delta_{h(s)} u - \Delta_{h_0} u - D_1(u) s - \frac{D_2(u)}{2} s^2 } \le \frac{s^3}{6} \max_{s \in [0,2\delta]} \, \abs*{ \frac{d^3 \Delta_{h} u}{ds^3}(s)},
        \end{equation}
    where
        \[
            D_1(u) = \frac{d \Delta_{h}u}{ds} |_{s=0}, \qquad D_2(u) = \frac{d^2 \Delta_{h}u}{ds^2} |_{s=0}.
        \]
    
    Moreover, the fact that $h(s)$ is $C^3$ with respect to $s$ clearly implies that there exist positive constants $C_1$ and $C_2$ depending only on the dimension $n$, on $\delta$ and on the restriction to $[0,2\delta]$ of $h$ and its derivatives such that for any $i=1,2,3$ it holds
    \begin{equation} \label{eq:stima-hessiano-s}
    \max_{s \in [0,2\delta]} \, \abs*{ \frac{d^i \Delta_{h} u}{ds^i}(s)} \le C_1 \max_{s \in [0,2\delta]} \Big( \abs{\nabla_{h(s)}^2 u}_{h(s)} + \abs{\nabla_{h(s)} u}_{h(s)} \Big) \le C_1 C_2 \Big( \abs{\nabla_{h_0}^2 u}_{h_0} + \abs{\nabla_{h_0} u}_{h_0} \Big),
    \end{equation}
    where in last inequality we used the fact that the difference of two connections is always a tensor and that on a compact manifold two Riemannian metrics are always equivalent. Using the change of variable $s = \eps r$ with the estimates~\eqref{eq:resto-Delta} and~\eqref{eq:stima-hessiano-s} we obtain
    \begin{equation} \label{eq:est-lap-ang}
        \abs*{ \Delta_{h(\eps r)} u - \Delta_{h_0} u - D_1(u) \eps r - \frac{D_2(u)}{2} \eps^2 r^2 } \le \frac{C_1 C_2}{6} \eps^3 r^3 \Big( \abs{\nabla_{h_0}^2 u}_{h_0} + \abs{\nabla_{h_0} u}_{h_0} \Big),
    \end{equation}
    which is valid for $r \le 2 \delta / \eps$. 

    We consider now a generic function $u \in C^2\big( \mathcal{C}(Y) \setminus \{ P\} \big)$, and from~\eqref{eq:Laplacian-1}, \eqref{eq:measure-3} and~\eqref{eq:est-lap-ang} we deduce that
    \begin{equation}
        \label{eq:eps-Delta}
        \Delta_{g_\eps}u = \Delta_{g_0}u + \eps \Delta_1 u + \eps^2 \Delta_2 u + \eps^3 \Delta_3 u,
    \end{equation}
    where
    \begin{equation}
        \label{eq:Delta-1-2}
        \Delta_1 u = \frac{1}{2} ( \mathrm{tr}_{h_0} \xi ) \partial_r u + \frac{D_1(u(r,\cdot))}{r}, \qquad 
        \Delta_2 u 
        =
        \frac{1}{2} \Big( \big(  \mathrm{tr}_{h_0} \eta -\abs{\xi}_{h_0}^2 \big) r \partial_r u + D_2(u(r, \cdot)) \Big), 
    \end{equation}
    and
    \begin{equation} \label{eq:resto-Delta-2}
        \abs{\Delta_3 u} \le C r \Big( r \abs{\nabla_{g_0} u}_{g_0}
     +  r^2 \abs{\nabla_{g_0}^2 u}_{g_0} \Big), 
    \end{equation}
    where to derive the previous estimate we used the following elementary inequalities:
    \begin{equation} \label{eq:link-to-conic}
        \abs{\partial_r u} \le \abs{\nabla_{g_0} u}_{g_0}, \qquad
        \abs{\nabla_{h_0} u(r, \cdot)}_{h_0} \le r \abs{\nabla_{g_0} u}_{g_0},
        \qquad
        \abs{\nabla_{h_0}^2 u(r, \cdot)}_{h_0} \le r^2 \abs{\nabla_{g_0}^2 u}_{g_0}.
    \end{equation}

    Similarly, for the scalar curvature and for any $s \le 2 \delta$ we have
    \begin{equation} \label{eq:resto-scal}
            \abs*{ R_{h(s)} - (n-1)(n-2) - R_1 s - \frac{R_2}{2} s^2 } \le \frac{s^3}{6} \max_{s \in [0,2\delta]} \, \abs*{ \frac{d^3 R_h}{ds^3}(s)},
        \end{equation}
    where $R_1=R_1(\xi)$ and $R_2=R_2(\xi, \eta)$ coincide respectively with the right-hand side of~\eqref{eq:scal-derivata-prima} evaluated at $s=0$ and with the right-hand side of~\eqref{eq:der-scalar-2}. From the estimate~\eqref{eq:resto-scal} and the formula~\eqref{eq:curv-scal-1} it follows that
    \begin{equation}
        \label{eq:eps-Scal}
        R_{g_\eps} = \eps \widehat{R}_1 + \eps^2 \widehat{R}_2 + \eps^3 \widehat{R}_3,
    \end{equation}
    where
    \begin{equation}
        \label{eq:hat-scal}
        \widehat{R}_1 = \frac{R_1 - n \mathrm{tr}_{h_0} \xi}{r}, \qquad
    \widehat{R}_2 =
    \frac{R_2}{2} +
    \bigg( \! n + \frac{3}{4} \bigg) \abs{\xi}_{h_0}^2 - \frac{1}{4} \big( \mathrm{tr}_{h_0} \xi \big)^2 - (n+1) \mathrm{tr}_{h_0} \eta,
    \end{equation}
    and with $\abs{\widehat{R}_3} \le C r$, for every $r \le 2 \delta / \eps$.

    At this point, formulas~\eqref{eq:exp-Lapl} and~\eqref{eq:L0} are an easy consequence of~\eqref{eq:eps-Delta} and~\eqref{eq:eps-Scal}. 
    Combining~\eqref{eq:stima-hessiano-s},~\eqref{eq:Delta-1-2},~\eqref{eq:link-to-conic} and~\eqref{eq:hat-scal} we deduce the estimates in~\eqref{eq:L3} for $i=1,2$, similarly the case $i=3$ follows from~\eqref{eq:resto-Delta-2} and the bound on $\widehat{R}_3$. Using the identities~\eqref{eq:Delta-1-2}, \eqref{eq:hat-scal} and the fact that $D_1 = D_2 = 0$ for radial functions we obtain the formulas in~\eqref{eq:L1} and~\eqref{eq:L2U}.

    We are left with the proof of  property~\eqref{eq:L1-auto}. Let $u$ and $v$ be smooth functions with compact support on $\mathcal{C}(Y) \setminus \{P \}$: then
        \[
        \int_{\mathcal{C}(Y)} u \lp_{g_\eps} v \, d \mu_{g_\eps} =
        \int_{\mathcal{C}(Y)} v \lp_{g_\eps} u \, d \mu_{g_\eps},
        \]
    as a consequence of the fact that $\lp_{g_\eps}$ is self-adjoint in $L^2(\mathcal{C}(Y), d \mu_{g_\eps})$, for every $\eps \ge 0$. On the other hand, from~\eqref{eq:exp-mu} and~\eqref{eq:exp-Lapl} we deduce that
        \[
        \int_{\mathcal{C}(Y)} u \lp_{g_\eps} v \, d \mu_{g_\eps} =
        \int_{\mathcal{C}(Y)} u \lp_0 v \, d \mu_{g_0}
        +
        \eps 
        \int_{\mathcal{C}(Y)} u \big( \lp_1 + m_1 \lp_0 \big) v \, d \mu_{g_0} + o(\eps),
        \]
    and the same is true exchanging $u$ and $v$, the conclusion follows.
\end{proof}

\smallskip

We now pass to the expansion of the Yamabe quotient on the conical manifold $(M,\bar{g})$, starting with 
the denominator in \eqref{eq:Qg}. To this aim, we introduce a cut-off function $\chi_\delta(s,y)$ defined on the cone $(\mathcal{C}(Y), g)$ and depending only on the $g$-distance from its vertex, which in particular coincides with $s$. This function will allow us to define global test functions on the singular manifold $M$ starting from the functions $u_\eps$ defined by~\eqref{eq:def-test-functions}. More precisely, we choose $\chi_{\delta}(s)$ with the following properties
\begin{equation}\label{eq:cut-off}
    \begin{cases}
    \displaystyle
        \chi_{\delta}(s) = 1 & \text{for $s \le \delta$}, \\[0.5ex]
        \chi_{\delta}(s) = 0 & \text{for $s \ge 2 \delta$}, \\[0.5ex]
        \abs{\nabla_g^k \hspace{0.02cm} \chi_{\delta}}_g (s) \leq C \delta^{-k}, 
        & \text{for $ \delta \le s \le 2\delta$, and $k=1,2$},
    \end{cases}
\end{equation}
for some positive constant $C$.

Let $\mathcal{U}$ be a neighborhood of the conical point $P$ in $M$ with the property that there exists a diffeomorphism $\sigma$ between $\mathcal{U} \setminus \{ P \}$ and $(0,R) \times Y$, for a certain $R > 0$, such that
    \begin{equation} \label{eq:def-sigma}
        \sigma_{*} \bar{g} = g = ds^2 + s^2 h(s),
    \end{equation}
see \eqref{eq:diffeo}.    
Finally, we can define our family of test functions on $M$, for $\delta < R/2$ and $\eps > 0$ we set
    \begin{equation} \label{eq:global-test-functions}
    \widehat{u}_{\eps, \delta}(x) =
        \begin{cases}
        \displaystyle
            \chi_\delta (\sigma(x)) u_\eps (\sigma(x)) & \text{if $x \in \mathcal{U}$} \\[0.5ex]
            0 & \text{otherwise in $M$}.
        \end{cases}
    \end{equation}

In order to find proper cancellations,
we will crucially use the property of $\Psi$ of being a solution of 
\[
\widetilde{\lp} \Psi = - \lp_1 U,
\]
see~\eqref{eq:equation-Psi},~\eqref{eq:FU} and the discussion at the beginning of Subsection~\ref{ss:c-psi}.

\begin{lemma}\label{l:den}
In the above notation, for $n \geq 5$, we have that 
\[
\int_{M }  \abs{\widehat{u}_{\eps, \delta}}^{\frac{2n}{n-2}} \, d\mu_{\bar{g}} = 1 + \eps^2 \int_{\mathcal{C}(Y)} \bigg( \frac{n(n+2)}{(n-2)^2} \Psi^2 U^{\frac{4}{n-2}} + \frac{2n}{n-2} \Psi U^{\frac{n+2}{n-2}} m_1 + U^{\frac{2n}{n-2}} m_2 \bigg) \, d\mu_{g_0} + o(\eps^2),
\]
where $m_1$ and $m_2$ are defined respectively by~\eqref{eq:m1} and by~\eqref{eq:m2}.
\end{lemma}

\begin{proof}

From the definition of $\widehat{u}_{\eps, \delta}$ given by~\eqref{eq:global-test-functions} it is clear that 
    \begin{equation*} 
        \int_{M}  \abs{\widehat{u}_{\eps, \delta}}^{\frac{2n}{n-2}} \, d\mu_{\bar{g}} 
        =
        \int_{\mathcal{U}}  \abs{\widehat{u}_{\eps, \delta}}^{\frac{2n}{n-2}} \, d\mu_{\bar{g}}
        =
        \int_{\{s \le 2 \delta \}}  \abs{\chi_\delta (s) u_\eps(s,y) }^{\frac{2n}{n-2}} \, d\mu_{g}.
    \end{equation*}
Moreover, from the discussion at the beginning of Section~\ref{s:exp-I} it follows that
    \begin{equation*} 
        \int_{\{s \le 2 \delta \}}  \abs{\chi_\delta (s) u_\eps(s,y) }^{\frac{2n}{n-2}} \, d\mu_{g}
        =
        \int_{\{r \le 2 \delta / \eps \}}  \abs{\chi_\delta (\eps r) w_\eps(r,y) }^{\frac{2n}{n-2}} \, d\mu_{g_{\eps}},
    \end{equation*}
where $w_\eps$ is defined by~\eqref{eq:w-eps}.

We will first show that the estimate holds for the right-hand side of the previous identity in the region $\{ r \le \delta / \eps \}$, where the cut-off function $\chi_\delta(\eps r)$ is equal to one, and then we prove that the contribution from the region $\{ \delta \le s \le 2\delta\}$ is negligible.

We focus our attention on the set $\{r \leq \delta/\eps\}$. From~\eqref{eq:cont-Psi-est} there exists $C > 0$ such that
    \[
        \eps \abs{\Psi} \leq C (\eps + \delta) U, \ \text{in $\{r \le \delta/ \eps \}$.} 
    \]
This shows that $w_\eps = U + \eps \Psi > 0$ in this region 
if $\delta$ is taken sufficiently small. Using again~\eqref{eq:cont-Psi-est} together with the elementary formula 
 $(1+t)^{\alpha} = 1 + \alpha t + \alpha(\alpha - 1) t^2/2 + O(t^3)$ for $t \to 0$, where $\alpha = 2n/(n-2)$, we deduce that there exist positive constants $\eps_0$ and $C=C(\delta)$ such that for any $\eps \le \eps_0$, for any $r \le \delta / \eps$ and for any $y \in Y$ it holds
\begin{equation*} 
    \abs*{w_\eps^{\frac{2n}{n-2}} - U^{\frac{2n}{n-2}} - \frac{2n}{n-2} \eps \Psi U^{\frac{n+2}{n-2}} - \frac{n(n+2)}{(n-2)^2} \eps^2 \Psi^2 U^{\frac{4}{n-2}}} \le C (1+r) \Psi^2 U^{\frac{4}{n-2}} \eps^3.
\end{equation*} 
Integrating this quantity and recalling that $U$ has $L^{\frac{2n}{n-2}}$-norm  equal to one, we deduce that
\begin{align*}
    \notag
    \int_{\{r \leq \delta/\eps\}} w_{\eps}^{\frac{2n}{n-2}} \, d\mu_{g_\eps} = & 1 
    + \eps \int_{\mathcal{C}(Y)} \bigg( \frac{2n}{n-2} \Psi U^{\frac{n+2}{n-2}} + U^{\frac{2n}{n-2}} m_1 \bigg) \, d\mu_{g_0} \\[1.5ex]
    & + \eps^2 \int_{\mathcal{C}(Y)} \bigg( \frac{n(n+2)}{(n-2)^2} \Psi^2 U^{\frac{4}{n-2}} + \frac{2n}{n-2} \Psi U^{\frac{n+2}{n-2}} m_1 + U^{\frac{2n}{n-2}} m_2 \bigg) \, d\mu_{g_0} + o(\eps^2), 
\end{align*}
where we used~\eqref{eq:exp-mu} and the fact that the terms integrated over $\mathcal{C}(Y)$ are of class $L^1$. 
We next notice that the coefficient of $\eps$ in the above formula vanishes 
because $U$ is radial, $m_1$ has zero angular average and because of~\eqref{eq:ortho-Psi}. 

For $\delta$ small (but fixed), similarly to before we have that $\eps \abs{\Psi} \leq C(\eps + \delta) U$ in $\{ \delta /\eps \le r \le 2 \delta / \eps \}$, 
which by~\eqref{eq:exp-mu}-\eqref{eq:m2} implies
\[
    \int_{\{ \delta / \eps \le r \le 2 \delta / \eps \}}  \abs{\chi_\delta (\eps r) w_\eps(r,y) }^{\frac{2n}{n-2}} \, d\mu_{g_{\eps}}
    \le
    C 
    \int_{\delta/\eps}^{2 \delta/\eps} U(r)^{\frac{2n}{n-2}} r^{n-1} \, dr
    \le C \delta^{-n} \eps^n.
\]
This concludes the proof.
\end{proof}

\begin{remark}\label{r:den-4}
    The estimates in the previous lemma also work when $n = 4$: however in this    dimension we will use a different choice of $\Psi$, so it will be convenient to prove their counterpart separately. 
\end{remark}

\smallskip

We now turn to the estimate of the numerator in \eqref{eq:Qg}.

\begin{lemma}\label{l:num}
     In the above notation, for $n \geq 5$, we have that  
    \begin{align*}
    \notag
    \int_{M} \Big( a \abs{\nabla_{\bar{g}} \hspace{0.02cm} \widehat{u}_{\eps,\delta}}^2_{\bar{g}} + R_{\bar{g}} \hspace{0.02cm} \widehat{u}_{\eps,\delta}^2 \Big) \, d\mu_{\bar{g}} = \mathcal{Y}_P 
    & + \eps^2 \bigg( \int_{\mathcal{C}(Y)} \Psi \lp_0 \Psi \, d\mu_{g_0} + 2 \int_{\mathcal{C}(Y)} \Psi (\lp_1 + m_1 \lp_0) U \, d\mu_{g_0} \bigg) \\[1.5ex]
    & + \eps^2 \bigg( \int_{\mathcal{C}(Y)} U (\lp_2 + m_1 \lp_1 + m_2 \lp_0) U \, d\mu_{g_0} \bigg) + o(\eps^2),
\end{align*} 
where the operators $\lp_i$ and the functions $m_i$ are defined as in Lemma~\ref{l:errori-precisi}.
\end{lemma}

\begin{proof}
As for the case of a smooth closed manifold we have that
    \[
        \int_{M} \Big( a \abs{\nabla_{\bar{g}} \hspace{0.02cm} \widehat{u}_{\eps,\delta}}^2_{\bar{g}} + R_{\bar{g}} \hspace{0.02cm} \widehat{u}_{\eps,\delta}^2 \Big) \, d\mu_{\bar{g}}
        =
        \int_{M} \widehat{u}_{\eps,\delta} \lp_{\bar{g}} \widehat{u}_{\eps,\delta} \, d\mu_{\bar{g}},
    \]
where $\lp_{\bar{g}}$ is the conformal Laplacian on the conical manifold $(M,\bar{g})$, for a proof see Proposition~3.2 in~\cite{AKM15}.
As before, from the definition of $\widehat{u}_{\eps,\delta}$ and the discussion at the beginning of Section~\ref{s:exp-I} we obtain
    \begin{equation} \label{eq:L44-1}
        \int_{M} \widehat{u}_{\eps,\delta} \lp_{\bar{g}} \widehat{u}_{\eps,\delta} \, d\mu_{\bar{g}} = 
        \int_{\{ r \le \delta / \eps \}} w_{\eps} \lp_{g_{\eps}} w_{\eps} \, d\mu_{g_{\eps}}
        +
        \int_{\{ \delta / \eps \le r \le 2\delta / \eps \}} \check{w}_{\eps,\delta} \lp_{g_{\eps}} \check{w}_{\eps,\delta} \, d\mu_{g_{\eps}},
    \end{equation}
where for notation convenience we set $\check{w}_{\eps,\delta}(r,y) = \chi_{\delta}(\eps r) w_{\eps}(r,y)$.

We start from the region $\{r \leq \delta/\eps\}$. Recalling the definition of $w_\eps$ in~\eqref{eq:w-eps}, using Lemma~\ref{l:bd-Psi} and Lemma~\ref{l:errori-precisi} we find that 
\begin{align}
    \notag
    \int_{\{r \leq \delta/\eps\}} w_\eps \lp_{g_\eps} w_\eps \, d\mu_{g_\eps} = & \int_{\mathcal{C}(Y)}  U \lp_0 U \, d\mu_{g_0}
    +
    \eps \bigg( \int_{\mathcal{C}(Y)} \Psi \lp_0 U \, d\mu_{g_0} + \int_{\mathcal{C}(Y)} U \lp_0 \Psi \, d\mu_{g_0} \bigg)
    \\[1.5ex]
    \notag
    & + \eps \bigg( \int_{\mathcal{C}(Y)} U (\lp_1 + m_1 \lp_0) U \, d\mu_{g_0} \bigg)
    + \eps^2 \bigg( \int_{\mathcal{C}(Y)} \Psi \lp_0 \Psi \, d\mu_{g_0} \bigg)
    \\[1.5ex]
    \notag
    & + \eps^2 \bigg( \int_{\mathcal{C}(Y)} \Psi (\lp_1 + m_1 \lp_0) U \, d\mu_{g_0}
    + \int_{\mathcal{C}(Y)} U (\lp_1 + m_1 \lp_0) \Psi \, d\mu_{g_0}
    \bigg) \\[1.5ex]
    \label{eq:expansion-numerator}
    & + \eps^2 \bigg( \int_{\mathcal{C}(Y)} U (\lp_2 + m_1 \lp_1 + m_2 \lp_0) U \, d\mu_{g_0} \bigg) + o(\eps^2),
\end{align}
where we used the fact that all integrands are of class $L^1$, from the 
dimensional assumption and the decay of $U$ and its derivatives. 

We claim now that each term in the coefficient of $\eps$ vanishes. 
From~\eqref{eq:ortho-Psi} we deduce the vanishing of $\int \Psi \lp_0 U \, d\mu_{g_0}$, and the same holds for $\int U \lp_0 \Psi \, d\mu_{g_0}$ because $\lp_0$ is self-adjoint in $L^2(d\mu_{g_0})$. By~\eqref{eq:L1copy} and~\eqref{eq:xi-zero-mean} we also get 
$\int U \lp_1  U \, d\mu_{g_0} = 0$. Moreover, 
the fact that $U$ is radial and that the angular average of $m_1$ vanishes implies that $\int m_1 U \lp_0  U \, d\mu_{g_0} = 0$. 

We also notice that 
\[
\int_{\mathcal{C}(Y)} \Psi (\lp_1 + m_1 \lp_0) U \, d\mu_{g_0}
    = \int_{\mathcal{C}(Y)} U (\lp_1 + m_1 \lp_0) \Psi \, d\mu_{g_0}, 
\]
because we know from~\eqref{eq:L1-auto} that the operator $\lp_1 + m_1 \lp_0$ is self-adjoint with respect to $L^2(\mathcal{C}(Y),\mu_{g_0})$. 

With the last considerations, we proved the desired estimate for the first addend in the right-hand side of~\eqref{eq:L44-1}. To estimate the second addend in the right-hand side of~\eqref{eq:L44-1}, we notice that by Lemma \ref{l:bd-Psi} there exists $C=C(\delta) > 0$ such that 
    \begin{equation*}
    \abs{w_\eps} \leq C \eps^{n-2}, \quad  \abs{\nabla_{g_0} w_\eps}_{g_0} \leq C \eps^{n-1}, \quad \abs{\nabla^2_{g_0} w_\eps}_{g_0} \leq C \eps^{n}, 
    \quad \text{in $\{ \delta / \eps \le r \le 2\delta / \eps \}$}.
    \end{equation*}
  Therefore, using the properties~\eqref{eq:cut-off} of $\chi_{\delta}$, the Leibnitz rule, and Lemma~\ref{l:errori-precisi}, we obtain  
\begin{equation} \label{eq:annulus}
\abs*{ \int_{\{ \delta / \eps \le r \le 2\delta / \eps \} } \check{w}_{\eps, \delta} \lp_{g_\eps} \check{w}_{\eps, \delta} \, d \mu_{g_\eps}} \leq C \hspace{0.03cm} \mathrm{Vol}_{g_\eps} (A_{\eps,\delta}) \eps^{n-2} \eps^{n} 
\leq C \eps^{n-2} = o(\eps^2),
\end{equation}
where we set $ A_{\eps, \delta} := \{ \delta / \eps \le r \le 2\delta / \eps \}$. This gives the desired assertion.
\end{proof}

We have then the following result. 

\begin{proposition} \label{p:upper-bd-Y}
 For $n \ge 5$, let $\widehat{u}_{\eps, \delta}$ be as above, and 
 suppose we are under the assumptions of Theorem \ref{t:main}. 
 Then, for $\eps > 0$ sufficiently small, we have that 
 \[
 Q_{\bar{g}}(\widehat{u}_{\eps, \delta}) < \mathcal{Y}_P. 
 \]
\end{proposition}

\begin{proof}
From Lemma~\ref{l:den} we have that 
 \begin{multline}
    \label{eq:expansion-denominator-2}
    \bigg( \int_{M }  \abs{\widehat{u}_{\eps, \delta}}^{\frac{2n}{n-2}} \, d\mu_{\bar{g}} \bigg)^{-\frac{n-2}{n}} \\
    = 1
    - \eps^2 \int_{\mathcal{C}(Y)} \bigg( \frac{n+2}{n-2} \Psi^2 U^{\frac{4}{n-2}} + 2 \Psi U^{\frac{n+2}{n-2}} m_1 + \frac{n-2}{n} U^{\frac{2n}{n-2}} m_2 \bigg) \, d\mu_{g_0} + o(\eps^2),
\end{multline}
where we used the elementary formula $(1+t)^{\alpha} = 1 + \alpha t + o(t)$ for $t \to 0$, with $\alpha = - (n-2)/n$. 

From Lemma~\ref{l:num} and~\eqref{eq:expansion-denominator-2} it follows immediately that the term of order $\eps$ in the expansion of $Q_{\bar{g}}(\widehat{u}_{\eps, \delta})$ vanishes, while the term of order $\eps^2$ 
is given by 
\begin{multline}
    \label{eq:expansion-3}
    \int_{\mathcal{C}(Y)} \Psi \lp_0 \Psi \, d\mu_{g_0} + 2 \int_{\mathcal{C}(Y)} \Psi \lp_1 U \, d\mu_{g_0}
    + \int_{\mathcal{C}(Y)} U (\lp_2 + m_1 \lp_1 ) U \, d\mu_{g_0} \\[1ex]
    - \mathcal{Y}_P \int_{\mathcal{C}(Y)} \bigg( \frac{n+2}{n-2} \Psi^2 U^{\frac{4}{n-2}} -\frac{2}{n} U^{\frac{2n}{n-2}} m_2 \bigg) \, d\mu_{g_0},
\end{multline}
where we also take into account the fact that $\lp_0 U = \mathcal{Y}_P U^{\frac{n+2}{n-2}}$.

Let $\Psi$ be as above, namely the  solution of~\eqref{eq:equation-Psi} and~\eqref{eq:FU}: with this choice, we can rewrite the quantity in~\eqref{eq:expansion-3} as follows
\begin{equation}
    \label{eq:expansion-4}
    - \int_{\mathcal{C}(Y)} \Psi \widetilde{\lp} \Psi \, d\mu_{g_0} + \int_{\mathcal{C}(Y)} U (\lp_2 + m_1 \lp_1 ) U \, d\mu_{g_0} 
    + \frac{2 \mathcal{Y}_P}{n} \int_{\mathcal{C}(Y)}  U^{\frac{2n}{n-2}} m_2 \, d\mu_{g_0}.
\end{equation}

From~\eqref{eq:m1}, \eqref{eq:L1}, \eqref{eq:relazioni-integrali-bolla-2} and integrating by parts we obtain
\begin{equation}
    \label{eq:III-simplified}
    \int_{\mathcal{C}(Y)} U m_1 \lp_1 U \, d\mu_{g_0} = \omega \int_Y \bigg( \frac{1}{2} \abs{\nabla \mathrm{tr}  \xi}^2 - \frac{(n-1)(n-4)}{2(n-2)} ( \mathrm{tr}  \xi )^2 - \frac{1}{2} \langle \delta \xi , d ( \mathrm{tr}  \xi ) \rangle \bigg) \, d \mu_{h_0},
\end{equation}
where $\omega$ is the constant given by \eqref{eq:cost-omega}. 
On the other hand, from~\eqref{eq:m2} and~\eqref{eq:relazioni-integrali-bolla-1} we have
\begin{equation}
    \label{eq:IV-simplified}
    \frac{2 \mathcal{Y}_P}{n} \int_{\mathcal{C}(Y)}  U^{\frac{2n}{n-2}} m_2 \, d\mu_{g_0} = \frac{n(n-4)}{n-2} \omega \int_Y \bigg( \frac{1}{4} ( \mathrm{tr}  \xi )^2 - \frac{1}{2} \abs{\xi}^2 + \frac{1}{2}  \mathrm{tr} \hspace{0.02cm} \eta \bigg) \, d \mu_{h_0}.
\end{equation}
From~\eqref{eq:L2U},~\eqref{eq:der-scalar-2} and~\eqref{eq:relazioni-integrali-bolla-2} we also find  
\begin{align}
    \notag
    \int_{\mathcal{C}(Y)} U \lp_2 U \, d\mu_{g_0} = & - \frac{n(n-1)}{n-2} \omega \int_Y \Big( \abs{\xi}^2 - \mathrm{tr} \hspace{0.02cm} \eta \Big) \, d \mu_{h_0} + n \omega \int_Y \Big( \abs{\xi}^2 - \mathrm{tr}  \hspace{0.02cm} \eta \Big) \, d \mu_{h_0} \\[1.5ex]
    \notag
    & + \omega \int_Y \bigg( \! - \frac{1}{4} \abs{\nabla \xi}^2 
        + \frac{1}{2} \abs{\delta \xi}^2
        - \frac{1}{4} \abs{\nabla \mathrm{tr} \xi}^2 
        + \frac{n-2}{2} \abs{\xi}^2
        + \frac{1}{2} \langle A_{h_0}(\xi), \xi \rangle
        \bigg) \, d \mu_{h_0} \\[1.5ex]
        \label{eq:II-simplified}
    & - \omega \int_Y \frac{n-2}{2} \mathrm{tr} \hspace{0.02cm} \eta \, d\mu_{h_0} + \omega \int_Y \bigg( \frac{3}{4} \abs{\xi}^2 - \frac{1}{4} ( \mathrm{tr}  \xi )^2 - \mathrm{tr} \hspace{0.02cm} \eta \bigg) \, d\mu_{h_0}.
\end{align}
Putting together~\eqref{eq:III-simplified}, \eqref{eq:IV-simplified}, \eqref{eq:II-simplified} and Proposition~\ref{prop:second-var-EH}, we obtain that the expression in 
\eqref{eq:expansion-3} becomes 
\[
- \int_{\mathcal{C}(Y)} \Psi \widetilde{\lp} \Psi \, d\mu_{g_0}
                + \frac{\omega}{2} \mathrm{Vol}_{h_0}(Y)^{\frac{n-3}{n-1}} \mathcal{R}''(h_0)[\xi,\xi]
                + \frac{\omega}{4} \int_Y \Big( (\mathrm{tr} \xi )^2 - \abs{\xi}^2 \Big) \, d \mu_{h_0},
\]
which by Proposition \ref{p:neg-def} is negative-definite. 
\end{proof}

Finally, as anticipated in the introduction, the proof of Theorem~\ref{t:main} follows by combining the previous proposition with Theorem~1.12 in~\cite{ACM}.

%% file: Section4B.tex
\subsection{The four-dimensional case} \label{ss:n=4}

This subsection is devoted to the four-dimensional case, that leads to 
diverging integral quantities and to an extra $\abs{\log \eps}$ factor in the second order term in the expansion of the Yamabe quotient on a suitable modification of the functions $\widehat{u}_{\eps,\delta}$ defined by~\eqref{eq:global-test-functions}.

Again, in order to better control the error terms, we need to construct a sufficiently accurate approximate solution to the Yamabe equation, which is done by linearizing  $\lp_{g_\eps} u = \mathcal{Y}_P u^{\frac{n+2}{n-2}}$  near the standard bubble $U$. The idea here is that, since the 
leading orders arise from large values of the geodesic distance $r$ from $P$ 
in the metric $g_\eps$,  
in the linearized operator $\widetilde{\lp}$ on $\Psi$, see \eqref{eq:equation-Psi}-\eqref{eq:L1copy}, we may discard the potential term given by $U^{\frac{4}{n-2}}$ (this would not be possible in dimension $n \geq 5$, as it would produce errors of order $O(\eps^2)$). We recall that $\Psi = \Psi_1 + \Psi_2$, with $\Psi_1$ given by \eqref{eq:explicit-Psi1} and $\Psi_2$ 
solving \eqref{eq:eq-Psi2}. Recalling also the functions $f$ and $G$ from \eqref{eq:decomposition-xi}, 
we set for convenience 
\[
K = f - G. 
\]
From the expression of $U$ one has that 
for $r$ large (see \eqref{eq:parti-radiali})
\[
    q_2(r) = - 2 c_Y r^{-3} + O(r^{-5}) = 2 c_Y \Delta_{g_0}(r^{-1}) + O(r^{-5}),
\]
where $c_Y$ is the constant defined by~\eqref{eq:bolla-std}.

Let us consider a smooth radial function $e(r)$ such that 
\[
    e(r) =
    \begin{cases}
    \displaystyle
    r^{-1} & \text{for $r \ge 2$}, \\[0.5ex]
   0 & \text{for $r \le 1$}. 
\end{cases}
\]
If we define 
\begin{equation}\label{eq:hat-Psi-2}
 \widehat{\Psi}_2(r,y) := e(r) H(y), \ 
\text{where} \ - \Delta_{h_0} H + H = 2 a^{-1} \Delta_{h_0} K \ \text{on $Y$},
\end{equation}
then by a direct computation we derive 
    \begin{equation}\label{eq:delta-hat-Psi}
        - a \Delta_{g_0} \widehat{\Psi}_2 =
        a r^{-3} \big( - \Delta_{h_0} H + H \big)
        =
        2 r^{-3} \Delta_{h_0} K, \qquad \forall r \ge 2, \forall y \in Y.
    \end{equation}
Moreover, if $n=4$, then $a=4(n-1)/(n-2) = 6$.

We would like to prove now the counterparts of Propositions \ref{p:exp-full},  
\ref{p:neg-def} and \ref{p:upper-bd-Y}. Concerning the former two, due to logarithmic divergences, 
we cannot integrate on $\mathcal{C}(Y)$ for $r \in [0,\infty)$, and hence we will only consider first a corresponding 
section over the link $Y$. 

\begin{lemma}\label{l:quadr-4}
If $n = 4$ and if we are under the assumptions of  Theorem \ref{t:main}, for $r \geq 2$ the quadratic 
form 
    \begin{equation}
    \label{eq:expansion-55}
    J(\xi) : = - a r^4 \int_Y \abs{\nabla_{g_0} \widehat{\Psi}_2}_{g_0}^2 \, d\mu_{h_0}
    + \frac{1}{2} \mathrm{Vol}_{h_0}(Y)^{\frac{1}{3}} \mathcal{R}''(h_0)[\xi,\xi]
    + \frac{1}{4} \int_Y \Big( (\mathrm{tr} \xi )^2 - \abs{\xi}^2 \Big) \, d \mu_{h_0},
\end{equation}
  is negative-definite.     
\end{lemma}

The role of the quantity in \eqref{eq:expansion-55} will be clear from the 
proof of Proposition \ref{p:upper-bd-Y-4}, and we stress that the gradient 
of $\widehat{\Psi}_2$ is taken with respect to the full $g_0$-metric. We notice also that $J(\xi)$ is independent of $r$ as a consequence of the homogeneity of $\widehat{\Psi}_2$ when $r \geq 2$.

\begin{proof}
From~\eqref{eq:delta-hat-Psi}, for any $r \geq 2$ we have that 
    \begin{equation} \label{eq:L46-1}
    a r^4 \int_Y \abs{\nabla_{g_0} \widehat{\Psi}_2}_{g_0}^2 \, d \mu_{h_0}
    = a \int_Y \left( \abs{\nabla_{h_0} H}_{h_0}^2 + H^2 \right) \, d \mu_{h_0}. 
    \end{equation}
If, using Fourier decomposition on $Y$ with an orthonormal basis $\{ u_k \}_{k \in \mathbb{N}}$ of $L^2(Y, d\mu_{h_0})$ such that $-\Delta_{h_0} u_k = \lambda_k u_k$, we set
\[
H = \sum_{k=1}^{+\infty} d_k u_k, \quad K = \sum_{k=1}^{+\infty} c_k u_k,
\]
from the equation $- \Delta_{h_0} H + H =  2 a^{-1} \Delta_{h_0} K$ 
we must impose the relation 
\[ 
d_k = 2 a^{-1} \frac{\lambda_k}{1 + \lambda_k} c_k, 
\]
which implies 
\[
a \int_Y \big( \abs{\nabla_{h_0} H}_{h_0}^2 + H^2 \big) \, d\mu_{h_0} = 4 a^{-1} \sum_{k=1}^{+\infty} \frac{\lambda_k^2}{(1 + \lambda_k)} c_k^2.  
\]
Combining the Lichnerowicz's eigenvalue estimate ($\lambda_1 > 3$), which follows from our assumptions on $Y$, with the monotonicity of the function $t/(1+t)$ we find that
\[
4 a^{-1} \frac{\lambda_k^2}{(1 + \lambda_k)} c_k^2 
= \frac{2}{3} \frac{\lambda_k}{(1 + \lambda_k)} \lambda_k c_k^2 
>\frac 12 \lambda_k c_k^2,
\]
so, summing over $k$,  we deduce that 
\begin{equation} \label{eq:L46-2}
a \int_Y \left( \abs{\nabla_{h_0} H}_{h_0}^2 + H^2 \right) \, d\mu_{h_0} > \frac{1}2 \int_Y 
\abs{\nabla_{h_0} K}_{h_0}^2 \, d \mu_{h_0}. 
\end{equation}
We recall here that $c_0 = 0$ by \eqref{eq:xi-zero-mean}, and it follows from~\eqref{eq:hat-Psi-2} that also $d_0 = 0$. 

On the other hand, reasoning as for the proof of Proposition \ref{p:neg-def}, we have that 
\begin{equation} \label{eq:L46-3}
\frac{1}{2} \mathrm{Vol}_{h_0}(Y)^{\frac{1}{3}} \mathcal{R}''(h_0)[\xi,\xi]
    + \frac{1}{4} \int_Y \Big( (\mathrm{tr} \xi )^2 - \abs{\xi}^2 \Big) \, d \mu_{h_0}  
    \le \frac{1}2 \int_Y \abs{\nabla_{h_0} K}_{h_0}^2 \, d \mu_{h_0}. 
\end{equation}
Putting together~\eqref{eq:L46-1},~\eqref{eq:L46-2} and~\eqref{eq:L46-3} the conclusion follows. 
\end{proof}

We have next the counterpart of Proposition \ref{p:upper-bd-Y}. We define 
(see \eqref{eq:explicit-Psi1} and \eqref{eq:hat-Psi-2})
\[ 
\widehat{\Psi} = \Psi_1 + c_Y \widehat{\Psi}_2, \quad \widehat{w}_\eps(r,y) = U(r) + \eps \widehat{\Psi}(r,y), \quad \text{and} \quad \bar{u}_{\eps}(s,y) = \eps^{-\frac{n-2}{2}} \big( U(s/\eps) + \eps \widehat{\Psi}(s/ \eps, y) \big)
\]
and for a cut-off function $\chi_{\delta}$ as in \eqref{eq:cut-off}, 
we set also 
\begin{equation} \label{eq:global-test-functions-4}
    \widetilde{u}_{\eps, \delta}(x) =
        \begin{cases}
        \displaystyle
            \chi_\delta (\sigma(x)) \bar{u}_{\eps}(\sigma(x)) & \text{if $x \in \mathcal{U}$} \\[0.5ex]
            0 & \text{otherwise in $M$},
        \end{cases}
    \end{equation}
where $\mathcal{U}$ and $\sigma$ are defined as in~\eqref{eq:def-sigma}.

\begin{proposition} \label{p:upper-bd-Y-4}
 For $n \ge 4$, let $\widetilde{u}_{\eps, \delta}$ be as above, and 
 suppose we are under the assumptions of Theorem \ref{t:main}. 
 Then, for $\eps > 0$ sufficiently small, we have that 
 \[
 Q_{\bar{g}}(\widetilde{u}_{\eps, \delta}) < \mathcal{Y}_P. 
 \]
\end{proposition}

\begin{proof}
As for the final part in the proof of Proposition~\ref{p:upper-bd-Y}, we will prove that 
$Q_{\bar{g}}(\widetilde{u}_{\eps, \delta})$, starting from the denominator in the Yamabe quotient. 

Exactly as in the proof of Lemma~\ref{l:den}, see also Remark \ref{r:den-4}, one finds that 
\[
\int_{M }  \abs{\widetilde{u}_{\eps, \delta}}^4 \, d\mu_{\bar{g}} = 1 + \eps^2 \int_{\mathcal{C}(Y)} \bigg( 6 \widehat{\Psi}^2 U^2 + 4 \widehat{\Psi} U^3 m_1 + U^4 m_2 \bigg) \, d\mu_{g_0} + o(\eps^2),
\]
using the fact that, by the explicit definition of $\widehat{\Psi}$, also here all the 
integrands are of class $L^1$, and the coefficient of $\eps$ vanishes for the same reasons as before.

Therefore, as for~\eqref{eq:expansion-denominator-2} 
we get that 
\begin{equation}
    \label{eq:expansion-denominator-2-4}
    \bigg( \int_{M }  \abs{\widetilde{u}_{\eps, \delta}}^4 \, d\mu_{\bar{g}} \bigg)^{-\frac{1}{2}} = 1 + O(\eps^2). 
\end{equation}
We next turn to the expression in the numerator of $Q_{\bar{g}}(\widetilde{u}_{\eps, \delta})$. Similarly to Lemma~\ref{l:num} we have
\[
        \int_{M} \Big( a \abs{\nabla_{\bar{g}} \hspace{0.03cm} \widetilde{u}_{\eps,\delta}}^2_{\bar{g}} + R_{\bar{g}} \hspace{0.03cm} \widetilde{u}_{\eps,\delta}^2 \Big) \, d\mu_{\bar{g}}
        =
        \int_{\{ r \le \delta / \eps \}} \widehat{w}_{\eps} \lp_{g_{\eps}} \widehat{w}_{\eps} \, d\mu_{g_{\eps}}
        +
        \int_{\{ \delta / \eps \le r \le 2\delta / \eps \}} \widetilde{w}_{\eps,\delta} \lp_{g_{\eps}} \widetilde{w}_{\eps,\delta} \, d\mu_{g_{\eps}},
\]
where for notation convenience we set $\widetilde{w}_{\eps,\delta}(r,y) = \chi_{\delta}(\eps r) \widehat{w}_{\eps}(r,y)$.

Integrating first 
in the region $B_{\delta/\eps} := \{r \leq \delta/\eps\}$, where $\widehat{w}_{\eps} = U + \eps \widehat{\Psi}$, the counterpart of~\eqref{eq:expansion-numerator} is 
\begin{align}
    \notag
    \int_{B_{\delta/\eps}} \widehat{w}_\eps \lp_{g_\eps} \widehat{w}_\eps \, d\mu_{g_\eps} = & \int_{B_{\delta/\eps}}  U \lp_0 U \, d\mu_{g_0}
    +
    \eps \bigg( \int_{B_{\delta/\eps}} \widehat{\Psi} \lp_0 U \, d\mu_{g_0} + \int_{B_{\delta/\eps}} U \lp_0 \widehat{\Psi} \, d\mu_{g_0} \bigg)
    \\[1.5ex]
    \notag
    & + \eps \bigg( \int_{B_{\delta/\eps}} U (\lp_1 + m_1 \lp_0) U \, d\mu_{g_0} \bigg)
    + \eps^2 \bigg( \int_{B_{\delta/\eps}} \widehat{\Psi} \lp_0 \widehat{\Psi} \, d\mu_{g_0} \bigg)
    \\[1.5ex]
    \notag
    & + \eps^2 \bigg( \int_{B_{\delta/\eps}} \widehat{\Psi} (\lp_1 + m_1 \lp_0) U \, d\mu_{g_0}
    + \int_{B_{\delta/\eps}} U (\lp_1 + m_1 \lp_0) \widehat{\Psi} \, d\mu_{g_0}
    \bigg) \\[1.5ex]
    \label{eq:expansion-numerator-4}
    & + \eps^2 \bigg( \int_{B_{\delta/\eps}} U (\lp_2 + m_1 \lp_1 + m_2 \lp_0) U \, d\mu_{g_0} \bigg) + O(\eps^2),
\end{align}
where the control $O(\eps^2)$ on the error term 
can be obtained from~\eqref{eq:exp-mu},~\eqref{eq:exp-Lapl} and the decay of $\widehat{\Psi}$. 

Contrarily to~\eqref{eq:expansion-numerator}, see the comments after it, we cannot integrate on the whole $\mathcal{C}(Y)$, since some terms in \eqref{eq:expansion-numerator-4} are not of class $L^1$.  In fact, looking for example at the integrands in the coefficient of $\eps^2$, scaling-wise they all decay like $r^{-4}$, which is borderline non-integrable in four dimensions, but those where $\lp_0 U$ appears have faster decay due to \eqref{eq:lagrange-mult-U}.

By the decay of $U$ we have that $\int_{B_{\delta/\eps}}  U \lp_0 U \, d\mu_{g_0} = \mathcal{Y}_P + o(\eps^2)$. Moreover, all the terms appearing 
in the coefficient of $\eps$ vanish by angular cancellation, as for the proof of Lemma~\ref{l:num}. 

We turn next to the terms in \eqref{eq:expansion-numerator-4} that 
multiply $\eps^2$. Since the function $U m_2 \lp_0 U$ is of class $L^1(\mathcal{C}(Y), d\mu_{g_0})$, 
we are left with 
\begin{multline*}
    \eps^2 \bigg( \int_{B_{\delta/\eps}} \widehat{\Psi} \lp_0 \widehat{\Psi} \, d\mu_{g_0} +  \int_{B_{\delta/\eps}
} \widehat{\Psi} (\lp_1 + m_1 \lp_0) U \, d\mu_{g_0}
    + \int_{B_{\delta/\eps}
} U (\lp_1 + m_1 \lp_0) \widehat{\Psi} \, d\mu_{g_0}
    \bigg) \\[1ex]
   + \eps^2 \bigg( \int_{B_{\delta/\eps}
} U (\lp_2 + m_1 \lp_1 ) U \, d\mu_{g_0} \bigg) + O(\eps^2).
\end{multline*}
Recalling that $\widehat{\Psi}$ has inverse linear decay at infinity, while 
$U$ has inverse quadratic decay (with natural decays for their derivatives), with an integration by parts (producing boundary terms of order $O(\eps^2)$) we obtain that the latter quantity becomes 
\begin{equation}\label{eq:leading}
\eps^2 \bigg( \int_{B_{\delta/\eps}
} a \abs{\nabla_{g_0} \widehat{\Psi}}_{g_0}^2   \, d\mu_{g_0} +  2 \int_{B_{\delta/\eps}
} \widehat{\Psi} \lp_1 U \, d\mu_{g_0} + \int_{B_{\delta/\eps}
} U (\lp_2 + m_1 \lp_1 ) U \, d\mu_{g_0}
    \bigg) + O(\eps^2), 
\end{equation} 
where we used the fact that $R_{g_0} \equiv 0$, and also that $\widehat{\Psi} m_1 \lp_0 U$ 
is of class $L^1(\mathcal{C}(Y), d\mu_{g_0})$. 

We now notice that, for $n = 4$, 
\[
\frac{U'(r)}{n-2} + \frac{U(r)}{r} = O(r^{-5}). 
\]
This relation implies that for $r$ large
\[
\widehat{\Psi} = c_Y \widehat{\Psi}_2 + O(r^{-3}), \quad 
\nabla_{g_0} \widehat{\Psi} = c_Y \nabla_{g_0} \widehat{\Psi}_2 + O(r^{-4}),
\]
see~\eqref{eq:L1-2}-\eqref{eq:explicit-Psi1}, which yields 
\[
\int_{B_{\delta/\eps}} \abs{\nabla_{g_0} \widehat{\Psi}}_{g_0}^2 \, d\mu_{g_0} = c_Y^2 \int_{B_{\delta/\eps}} \abs{\nabla_{g_0} \widehat{\Psi}_2}_{g_0}^2 \, d\mu_{g_0}
+ O(\eps^2), 
\]
and together with 
$U(r) = c_Y r^{-2} + O(r^{-4})$ that 
\[
\lp_1  U = - 2 c_Y r^{-3}   \Delta_{h_0} 
 K   + O(r^{-5}), \quad \text{for $r$ large,}
\]
see~\eqref{eq:bolla-std} and~\eqref{eq:L1-2}, which implies 
\[
\int_{B_{\delta/\eps}
} \widehat{\Psi} \lp_1 U \, d\mu_{g_0} = - a \int_{B_{\delta/\eps}} \abs{\nabla_{g_0} \widehat{\Psi}}_{g_0}^2 \, d\mu_{g_0}  
+ O(\eps^2). 
\]
By these considerations, \eqref{eq:m1}, \eqref{eq:L1} and \eqref{eq:L2U} the expression in \eqref{eq:leading}
becomes, after integrating 
by parts on $Y$,
\begin{equation} \label{eq:interno-4}
 c_Y^2 \eps^2 \log \frac{\delta}{\eps} J(\xi)  + O(\eps^2),
\end{equation}
where $J(\xi)$ is as  in \eqref{eq:expansion-55}. Here we used the comments after Lemma \ref{l:quadr-4} 
and the fact that the region $\{r \leq 2\}$ contributes an order $O(\eps^2)$ to the integral. 

Concerning the region $A_{\eps,\delta} = \{ \delta/ \eps \le r \le 2 \delta / \eps \}$, similarly to 
\eqref{eq:annulus} we have that  
\begin{equation} \label{eq:annulus-2}
\abs*{ \int_{ A_{\eps,\delta} } \widetilde{w}_{\eps, \delta} \lp_{g_\eps} \widetilde{w}_{\eps, \delta} \, d \mu_{g_\eps}} \leq C \hspace{0.03cm} \mathrm{Vol}_{g_\eps} (A_{\eps,\delta}) \eps^{n-2} \eps^{n} 
\leq C \eps^{n-2} = O(\eps^2),
\end{equation}
where we used the fact that now $n = 4$. 

From~\eqref{eq:interno-4},~\eqref{eq:annulus-2} and~\eqref{eq:expansion-denominator-2-4} we finally obtain that 
\[
Q_{\bar{g}}(\widetilde{u}_{\eps, \delta}) = \mathcal{Y}_P +  c_Y^2 \eps^2 \log \frac{\delta}{\eps} J(\xi) + O(\eps^2) < \mathcal{Y}_P
\]
for $\eps$ small, by Lemma \ref{l:quadr-4}. This is the desired inequality. 
\end{proof}